\documentclass{article}
\usepackage[utf8]{inputenc}
\usepackage{appendix}
\usepackage{xcolor}
\newcommand{\Prob}{\mathbb{P}}
\newcommand{\Exp}{\mathbb{E}}
\newcommand{\R}{\mathbb{R}}
\newcommand{\C}{\mathbb{C}}
\newcommand{\Id}{\mathrm{Id}}
\newcommand{\iunit}{\mathbf{i}}
\newcommand{\cF}{\mathcal{F}}

\newcommand{\G}{\mathbf{G}}
\newcommand{\x}{\mathbf{x}}
\newcommand{\X}{\mathbf{X}}
\newcommand{\BA}{\mathbf{A}}

\usepackage{amsthm,amsmath,amssymb,amsfonts}
\DeclareMathOperator{\Tr}{Tr}
\DeclareMathOperator{\diff}{d}
\DeclareMathOperator{\sgn}{sgn}
\DeclareMathOperator{\Var}{Var}
\DeclareMathOperator{\rank}{rank}
\DeclareMathOperator{\Cov}{Cov}

\newtheorem{definition}{Definition}[section]
\newtheorem{theorem}{Theorem}[section]
\newtheorem{corollary}{Corollary}[section]
\newtheorem{proposition}{Proposition}[section]
\newtheorem{lemma}{Lemma}[section]
\theoremstyle{remark}
\newtheorem*{remark}{Remark}
\title{Central Limit Theorem for traces of the resolvents of “half-heavy tailed" Sample Covariance matrices}

\author{Svetlana Malysheva}
\date{\today}

\begin{document}

\maketitle
\begin{abstract}
    We consider the spectrum of the Sample Covariance matrix $\BA_N:= \frac{\X_N \X_N^*}{N}, $ where $\X_N$ is the $P\times N$ matrix with i.i.d. half-heavy tailed entries and $\frac{P}{N}\to y>0$ (the entries of the matrix have variance, but do not have the fourth moment). We derive the Central Limit Theorem for the Stieltjes transform of the matrix $\BA_N$ and compute the covariance kernel. Apart from that, we derive the Central Limit Theorem for the Stieltjes transform of overlapping Sample Covariance matrices. 
\end{abstract}
\section{Introduction}
% Introduce your model
In this paper, we will study the spectrum of Sample Covariance random matrices, i.e. the spectrum $\lambda_1, \dots, \lambda_P$ of matrices
\begin{equation}  \BA_N:=\frac{\X_N \X_N^{*}}{N},
\end{equation}
where
\begin{equation}
    X_N := \left(\x_1, \x_2, \dots, \x_N\right):=\left(x_{i j}\right)_{\substack{1\leq i\leq P \\ 1\leq j\leq N}}
\end{equation}
is a random $P\times N$ matrix  with i.i.d. entries. The dimensions of the matrix $\X_N$ will be assumed to be growing at the same speed, i.e. we will take the parameter $D\to +\infty$ such that
\begin{equation}
    \frac{P(D)}{N(D)} \to y\in \left(0, +\infty \right).
\end{equation}

When $\Exp\left|x_{i, j} \right|^4 < +\infty$ we will refer to this case as the “light-tailed" case, possibly without specifying all necessary conditions on  $x_{i, j}.$ 

In the case of our main interest $x_{i, j}$ will be taken to be regularly varying with parameter $\alpha$ such that $2<\alpha<4.$ This case will be called “half-heavy tailed". Particularly, in this case $\Exp\left|x_{i, j} \right|^2 < +\infty$ and $\Exp\left|x_{i, j} \right|^4 = +\infty.$

%Introduce LSS and how it's useful in stats 

For any function $\varphi $ defined on the spectrum of $\BA_N,$ the value 
\begin{equation}
    \theta^{\BA_N}_\varphi:= \frac{1}{P} \sum_{i=1}^P \varphi\left(\lambda_i\right)
\end{equation}
is called Linear Spectral Statistics (shortened as LSS), and $\varphi$ is called 
a test function. Central Limit Theorems for Linear Spectral Statistics usually refer to the pointwise by $\varphi$ convergence of re-normalised LSS to the Gaussian process with certain mean and covariance, where $\varphi$ belongs to some certain class of functions $\mathcal{D}.$ 
% the claim that a_n b_n is called CLT

Central Limit Theorems for linear spectral statistics can be used in hypothesis testing for high-dimensional data sets. If the number of observations of a certain high-dimensional random variable is proportional to the number of dimensions, CLT can be used to test if the covariance matrix of the random variable is the identity matrix. 
For example, when $\varphi(x):= x -\log x -1,$ $P\theta_{\phi}^{\BA_N}$ is called log-likelihood ratio statistics. For the “light-tailed case" it was shown in \cite{bai_corrections_2009} that when $P\sim N$, the test considering CLT (Corrected Likelihood Ratio Criterion) performs better than the traditional Likelihood Ratio Criterion (based on the convergence of $P\theta_{\phi}^{\BA_N}$ to $\chi_{\frac{1}{2} P(P+1)}^2$ when $P$ is fixed and $N \to + \infty$). Also, unlike previously existing tests (\cite{ledoit_hypothesis_2002}, \cite{srivastava_tests_2005}), it does not assume the normality of the data. More applications of CLTs to hypothesis testing can also be found in \cite[Chapter~9]{yao_large_2015}. Similarly, Central Limit Theorems for “half-heavy tailed" Sample Covariance matrices may have potential applications in testing high-dimensional data sets, whose entries are regularly varying random variables with exponent between 2 and 4. Such data sets include stock returns in short periods \cite{gabaix_theory_2003} and household incomes \cite{efthimiou_household_2016}.

For $\varphi(x):= \frac{1}{z-x},$ where $z\in \C \backslash \R,$
\begin{equation}     \theta^{\BA_N}_\varphi:= \frac{1}{P} \sum_{i=1}^P \frac{1}{z-\lambda_i}=\frac{1}{P} \Tr \left(z-\BA_N\right)^{-1},
     \label{eqn:reslss}
\end{equation}
and this way the trace of resolvent can also be considered as LSS. 
We will derive CLT for “half-heavy tailed" Sample Covariance matrices on the domain $\mathcal{D},$ that is a span of resolvent traces. In the “light-tailed case", it is possible to extend CLT (as it was done in \cite{bai_clt_2004}) from the resolvent traces to analytic functions on the open interval, containing the limiting spectrum. It was done by integration via contour, containing all the eigenvalues with probability converging to 1. The results \cite{bai_convergence_1988}, \cite{yin_limit_1988}, \cite{bai_limit_1993}, \cite{bai_no_1998} yield that it is possible to find such a contour. Nevertheless, in the “half-heavy tailed" case similar extension cannot be done the same way because the largest eigenvalues of $\BA_N$ tend to $+\infty$ (\cite{bai_note_1988}, \cite{auffinger_poisson_2009}).

Historically, the main directions of improvement of existing Central Limit Theorems for Linear Spectral Statistics included enlarging the class $\mathcal{D}$ of test functions and weakening the restrictions on the entries of the matrix $\X_N.$
%I need to discuss here what exactly Girko did because I am not sure. (Need to ask Boris) 
For the application of the CLT from \cite{bai_clt_2004}, the entries of $\X_N$ should have the first four moments matching the first four moments of the standard normal random variable, and the class $\mathcal{D}$ should consist of test functions that are analytic on the domain that contains the support of Marchenko-Pastur law. In \cite{lytova_central_2009}
the $4$th cumulant is not necessarily $0$ anymore, and the space of test functions is extended to those that have more than $5$ bounded derivatives. The term proportional to the $4$th cumulant was added. In the CLT provided in \cite{shcherbina_central_2011}, it is sufficient for the entries $\X_N$ to have at least $4+\epsilon$ moments, and the class of test functions is enlarged to functions $\varphi$ such that
\begin{equation}
\int(1+2|k|)^{3+\epsilon}|\widehat{\varphi}(k)|^2 d k < \infty.
\end{equation}

On the other hand, the CLT for LSS holds also for heavy-tailed random matrices i.e. those with $\alpha < 2$. In \cite{benaych-georges_central_2014}, the authors prove CLT for symmetric Lévy random matrices with infinite variance.

%4. summary of methodology for Wigner LSS, Maltsev-Benaych-Georges, Maltsev-Lodhia 5. short summary of your results 6. outline of the paper from before

Our methodology will be based on the methods used in \cite{BGM16}, \cite{LM22} and \cite{MM22}. In \cite{BGM16} CLT for resolvent traces of “half-heavy tailed" Wigner matrices were obtained and the covariance kernel was written in an integral form. The main steps of the proof include truncation, application of martingale CLT, removal of non-diagonal terms in the resolvents in the resulting formula and the computation of the covariance kernel using the approximation of $\mathbb{E}\left[\exp \left(-\mathbf{i}\left|\frac{\hat{x}_{i, j}}{\sqrt{N}}\right|^2 \lambda\right)\right]$ when $\Im \lambda <0.$ In \cite{LM22}, the integral in the expression obtained in \cite{BGM16} was evaluated explicitly by  separation of variables. Afterwards, the integral kernel associated with this covariance was extracted. Interestingly, the effect of the eigenvalues outside the limiting spectrum can be spotted in this kernel. 

As a first step, in this paper, we will prove the CLT for centred resolvent traces for “half-heavy tailed" Sample Covariance matrices, following \cite{BGM16}. Afterwards, similarly to \cite{LM22} we will simplify the expression obtained in the first step. Also, we will sketch the proof of CLT for the resolvent traces for overlapping “half-heavy tailed" Sample Covariance matrices. In the "light-tailed" case overlapping CLT was introduced in \cite{DP2018} for Sample Covariance matrices, and in \cite{borodin_clt_2014} for Wigner matrices; it got further development in \cite{soshnikov_2020}. With the methods in the first step, our overlapping CLT will not require the introduction of any new formulas or estimations, just extra attentiveness. Potential applications of the overlapping CLT include testing large-dimensional data sets with some missing data. 

The paper is organised as follows. In Section \ref{sec:mainproof} we will obtain CLT for the resolvent traces with limiting covariance in the form of an integral. In Subsection \ref{sec:truncation}, we will justify the truncation of the entries of $\X_N$ with the argument $N^{\frac{1}{4} + \frac{1}{\alpha} + \epsilon}$ for any $\epsilon>0.$ In Subsection \ref{sec:trunc_prop}, we will calculate the moments of truncated elements and approximate $\mathbb{E}\left[\exp \left(-\mathbf{i}\left|\frac{\hat{x}_{i, j}}{\sqrt{N}}\right|^2 \lambda\right)\right]$. In Subsection \ref{sec:mart_decomposition}, we will rewrite $\hat\theta_N(z)$ as a sum of martingale differences. It will then remain to prove the convergence of the expression, depending on the entries of matrix $\X_N$ and resolvents of $\BA_{N, k}$ (which are Sample Covariance matrices of $\X_N$ with one column removed). In Subsection \ref{sec:diagonalisation}, we will remove non-diagonal entries of the resolvents from this expression using moments calculated in Section \ref{sec:trunc_prop}. Here, we will use that $\alpha<4.$ Also, in this step we will need $\epsilon$ in the truncation argument to be small enough, such that inequality \eqref{eqn:epsilontochoose} holds. Finally, in Subsection \ref{sec:computation_of_limit}, we will use the approximation, computed in Subsection \ref{sec:trunc_prop}, to rewrite the expression from Section \ref{sec:diagonalisation} as an integral of $t$ and $s.$ We will upper-bound the absolute value of the function under the integral so as to justify our use of Dominated Convergence Theorem, and we compute its pointwise limit. In Section \ref{sec:covcalculation} we will calculate the integral expression, obtained for the limiting covariance in Section \ref{sec:mainproof}. In Section \ref{sec:overlapping} we will prove the CLT for overlapping matrices.

%Later here, I will need to mention papers with some population covariance matrix. In (Najim, Yao , 2 papers) it is positive with bonded norm. In (Zhou, Zheng, bai, Yao, Zhu) it is changed and AR(1) and AR(2) structures are tested.  

\section{Matrix Model, Notations and Main Results}
The random variables we will study in this paper have the following
distributional properties.

\begin{definition}[Regularly varying random variables]
\label{def:heavytails}
A real or complex random variable $Z$ is said to be heavy-tailed with 
parameter $\alpha$ if there exists a constant $c > 0$ and a 
slowly-varying function $\ell:\R^+\to\R^+$ for which
\[
\Prob\Big( \big\{ |Z| > x \big\} \Big) = \frac{-c
  \ell(x)}{\Gamma\left(1-\frac{\alpha}{2}\right) x^\alpha}, \qquad x> 0.
\]
Recall that a slowly-varying $\ell(x)$ satisfies
\[
\lim_{t\to\infty} \frac{\ell(tx)}{\ell(t)} = 1,
\]
for all $x \in \R^+$.
\end{definition}
\begin{remark}
To clarify, it is important to note that $\Gamma\left(1-\frac{\alpha}{2}\right)<0.$ We will focus solely on the function $\ell(\cdot)$, ensuring that its limit as $t$ approaches positive infinity is 1. We will select a suitable constant $c$ based on this criterion. The unconventional choice of $c$ stems from the intentional cancellation of $-\Gamma\left(1-\frac{\alpha}{2}\right)$ in Part 5 of Lemma \ref{lem:truncbound}, which subsequently manifests in the Main Results.
\end{remark}

We say ``half-heavy'' in this paper primarily because we will restrict
our study to the parameter ranges $2 < \alpha < 4$; the random matrix
literature often refers to the heavy-tail case as the parameter range
$0 < \alpha < 2$. We now construct the matrix model of interest in
this paper.

\begin{definition}[Half-Heavy tailed Sample Covariance Matrices]
\label{def:hlfhvywish}
Define a $P \times N$ matrix $\X_N$ whose entries are i.i.d random 
variables which are centred, have variance one and satisfy the tail 
decay of Definition~\ref{def:heavytails} with $2<\alpha<4$. Let
\[
\BA_N := \frac{\X_N\X_N^*}{N},
\]
be the Sample Covariance random matrix obtained from $X$.
\end{definition}

\begin{definition}[Empirical spectral measure]
For the $P\times P$ Hermitian matrix $\BA$ with eigenvalues $\lambda_1, \lambda_2, \dots, \lambda_P$ we define the empirical spectral measure as the measure $\mu_{\BA}$ with support on the real line such that  
$$
    \diff \mu_{\BA}(x):= \frac{1}{P}\sum_{i=1}^P \delta(x-\lambda_i) \diff x,
$$
    where $\delta(x)$ denotes Dirac delta function. 
\end{definition}

For matrices satisfying the properties of 
Definition~\ref{def:hlfhvywish}, the Marchenko-Pastur Law has been 
established, which we recall in the following Proposition (appearing as 
\cite[Theorem 3.6--7]{BS10})
\begin{proposition}
\label{prop:mplaw}
Let $\BA_N$ be as in Definition~\ref{def:hlfhvywish}, and suppose that $P$ 
and $N$ depend on an underlying parameter $D$ and tend to infinity in 
such a way that $\lim_{D\to\infty} P(D)/N(D) \to y \in (0,\infty)$. Let 
\begin{align*}
a &:= (1 - \sqrt{y})^2,\\
b &:= (1 + \sqrt{y})^2,
\end{align*}
then the empirical spectral measure of $\BA_N$, converges weakly 
almost surely to the Marchenko Pastur law, 
\[
\diff \mu_{\mathrm{MP},y}(x) =  p_y(x)\diff x  + 
\mathbf{1}_{y>1}\bigg(1 - \frac{1}{y}\bigg) \delta(x)\diff x
\]
where 
\[
p_y(x) =\frac{\sqrt{(b-x)(x-a)}}{2\pi x y} \mathbf{1}_{[a,b]}(x).
\]
\end{proposition}
For a probability measure $\mu(\cdot)$ with the support on the real line we define its Stieltjes transform in the point $z\in \C \backslash \R$ as 
\begin{equation}
    s_\mu(z):=\int_{\R}\frac{1}{z-\lambda} \diff \mu(\lambda).
\end{equation}

The resolvent matrix of $P\times P$ Hermitian matrix $\BA_N$ is defined as
\begin{equation}
\G_{\BA_N}(z) := (z\mathbf{\Id}_P - \BA_N)^{-1}, \qquad z \in \C\backslash \R.
\end{equation}
Noticeably, the renormalised trace of the resolvent of the matrix $\BA_N$ is equal to the Stieltjes transform of the empirical spectral distribution of the matrix $\BA_N:$
\begin{equation}
    s_{\BA_N}(z):=\int_{\R}\frac{1}{z-\lambda} \diff \mu_{\BA_N}(\lambda)=\sum_{i=1}^P\frac{1}{z-\lambda_i}= \frac{1}{P} \Tr \G_{\BA_N}(z).
\end{equation}
By \cite[Theorem~B.9]{BS10} the convergence of the probability measures on the real line is equivalent to the convergence of their Stieltjes transforms.  The Proposition \ref{prop:mplaw} is equivalent to the following statement holding. 

For random matrix $\BA_N$ satisfying Definition \ref{def:hlfhvywish} for each $z \in \C\backslash \R$
\[
\lim_{D\to\infty}\frac{1}{P} \Tr \G_{\BA_N}(z) = m_y(z) \qquad\hbox{almost 
surely},
\]
where $m_y(z)$ is the Stieltjes transform of $\mu_{\mathrm{MP},y}.$

The statement of the \cite[Lemma 3.11]{BS10} yields, that
\begin{equation}
\label{eqn:mpstieltj}
m_y(z) = \int_\R \frac{\mu_{\mathrm{MP},y}(\diff x)}{z-x} =  - \frac{z - 
(1 - y) - \sqrt{(z -1 - y)^2 - 4y}}{2yz}, 
\end{equation}
where the branch cut of the square root is taken on the 
positive real line.
It is easy to see, that  the Stieltjes transform of the Marchenko-Pastur law is the root of the following quadric equation:
\begin{equation}
\label{eqn:sttransformquadr}
    zym_y^2(z)+ (1-z-y)m_y(z)+1 =0.
\end{equation}

Our aim is to prove the following result regarding the fluctuation of 
the Stieltjes transform of $\BA_N$.
\begin{theorem}[Central Limit Theorem]
\label{thm:main}
Under the same assumptions of Proposition~\ref{prop:mplaw}, the linear 
spectral statistics 
\begin{equation}
	\theta_N(z)=\frac{1}{N^{1-\alpha/4}}(\Tr\G_{\BA_N}(z)-\mathbb{E}\Tr\G_{\BA_N}(z)).
\end{equation}
converges in distribution to a Gaussian process on $\C\backslash\R$ 
with covariance kernel 
\begin{equation}
\label{covkernel}
C(z, w)=\int_{t, s>0} \frac{\partial}{\partial z} 
\frac{\partial}{\partial w}\mathcal{L}(z, t, w, s) \diff t \diff s,
\end{equation}
where 
\begin{gather}
    \mathcal{L}(z, t, w, s):= y \exp\left(\iunit \sgn_{\Im z} t z+\iunit \sgn_{\Im w} sw\right) \times \ell(z, t, w, s) \times r(z, t, w, s)\\
    \ell(z, t, w, s):= c \frac{\left(K(z, t)+K(w, s)\right)^{\alpha/2}-K(z, t)^{\alpha/2}-K(w, s)^{\alpha/2}}{ts}\\
    r(z, t, w, s):= \exp\left(-yK(z, t) -yK(w, s)\right)
\end{gather}
where $\sgn_z$ is the sign of the imaginary part of $z$, 
$K(z,t)=it\sgn_z zm_{y}(z)$, and $m_{y}(z)$ is Stieltjes transform of 
Marchenko-Pastur law, equation~\eqref{eqn:mpstieltj}.
\end{theorem}

Similarly to the work \cite{LM22}, we will compute rewrite $\ell(z, t, w, s)$ in an integral form, that will enable us to separate $(z, t)$ and $(w, s).$ It will allow us to calculate $\int_{t, s>0} \frac{\partial}{\partial z} 
\frac{\partial}{\partial w}\mathcal{L}(z, t, w, s) \diff t \diff s$ explicitly. 

\begin{theorem}
\label{thm:covkernrewrite}
For all $z, w\in \C\backslash \R$ holds
\begin{multline*}
   C(z, w) = -yc \Gamma\left(1+\frac{\alpha}{2}\right)\frac{\partial}{\partial z}\left(zm_y(z)\right)\frac{\partial}{\partial w}\left(wm_y(w)\right)\times \\ \frac{\left(-1+zm_y(z)\right)^{\alpha/2-1}- \left(-1+wm_y(w)\right)^{\alpha/2-1}}{zm_y(z) - w m_y(w)}.
    \label{eqn:main2}
\end{multline*}
\end{theorem}

Minor adjustments to the proof Theorem \ref{thm:main} in allow us to calculate the Limiting eigenvalue statistics of overlapping half-heavy-tailed Sample Covariance matrices. 
We will sketch the proof and do the necessary computations to prove the following theorem.  

\begin{theorem}
Suppose that matrix $\X_N$ is as in the Theorem \ref{thm:main}. Denote the set of the row indices of the matrix $\X_N$ as $\mathcal{P}$ and the set of column indices as $\mathcal{Q},$ so that $\X_N=\left(x_{i, j}\right)_{\substack{i\in \mathcal{P} \\ j\in \mathcal{Q}}}.$
Choose $\mathcal{P}_i \subseteq \mathcal{P}$ and $\mathcal{Q}_i \subseteq \mathcal{Q}$ for $i = 1,\dots, d,$ where $d$ is not changing with $N.$ 
Take submatrices $\X_N^{[i]}:=\left(\x_{i, j}\right)_{\substack{i\in \mathcal{P}_i \\ j\in \mathcal{Q}_i}}.$ 
Let $\mathbf{A}^{[i]}_N := \frac{\X^{[i]}_N \X^{[i]*}_N}{N}$ and
\begin{equation}
    \theta^{[i]}_N(z):=\frac{1}{N^{1-\alpha/4}}\left(\Tr\G_{\mathbf{A}^{[i]}_N}(z)-\mathbb{E}\Tr\G_{\mathbf{A}^{[i]}_N}(z)\right).
    \end{equation}
    Assume also the following asymptotic: 
    \begin{equation}
    \frac{|\mathcal{P}_i|}{N} \underset{N \rightarrow \infty}{\rightarrow} p_i>0, \text{  } \frac{|\mathcal{Q}_i|}{N} \underset{N \rightarrow \infty}{\rightarrow} q_i>0, \text{  }, \frac{|\mathcal{P}_i\cap\mathcal{P}_j||\mathcal{Q}_i\cap\mathcal{Q}_j|}{N^2} \underset{N \rightarrow \infty}{\rightarrow} \gamma_{ij}>0. 
\end{equation}
Then for any $z_1, z_2, \dots, z_d \in \mathbb{C}\backslash\mathbb{R}$ a vector $\left<\theta_N^{[1]}(z_1), \theta_N^{[2]}(z_2), \dots, \theta_N^{[d]}(z_d)\right>$ converges to a complex Gaussian vector in distribution when $N\to\infty$ and 
\begin{equation}
    \Cov\left[\theta_N^{[i]}(z), \theta_N^{[j]}(w)\right] \underset{N \rightarrow \infty}{\rightarrow} C_{i, j}(z, w), 
\end{equation}
where
\begin{equation}
C_{i, j}(z, w)=\int_{t, s>0} \frac{\partial}{\partial z} 
\frac{\partial}{\partial w}\mathcal{L}^{[i, j]}(z, t, w, s) \diff t \diff s,
\end{equation}
and
\begin{gather}
    \mathcal{L}^{[i, j]}(z, t, w, s):= \gamma_{i, j} \exp\left(\iunit \sgn_{\Im z} t z+\iunit \sgn_{\Im w} sw\right) \times \ell(z, t, w, s) \times r(z, t, w, s),\\
    \ell^{[i, j]}(z, t, w, s):= c \frac{\left(K^{[i]}(z, t)+K^{[j]}(w, s)\right)^{\alpha/2}-K^{[i]}(z, t)^{\alpha/2}-K^{[j]}(w, s)^{\alpha/2}}{ts},\\
    r^{[i, j]}(z, t, w, s):= \exp\left(-p_iK^{[i]}(z, t) -p_jK^{[j]}(w, s)\right),
\end{gather}
with the notation $K^{[i]}(z, t) = t \sgn_z \iunit z \frac{1}{q_i}m_{\frac{p_i}{q_i}}\left(\frac{z}{q_i}\right).$
\label{thm:theoremforoverlapping}
\end{theorem}

\section{Proof of Theorem~\ref{thm:main}}
\label{sec:mainproof}
We follow the same approach as in \cite{BGM16,MM22}. First, we replace 
the matrix $\X_N$ with a centred and truncated version matrix $\hat \X_N$ 
and show that the linear spectral statistic (LSS) $\hat\theta_N(z)$ with
$\hat\X_N$ replacing $\X_N$ has the same distributional limit as $\theta_N(z)$.
We then apply the Martingale Central Limit Theorem to $\hat{\theta}_N(z)$,
using estimates we develop for the entries of $\hat\X_N$.

\subsection{Truncation}\label{sec:truncation}

In this section, we truncate the entries of $\X_N$ by setting them to zero when they exceed $N^\beta$ and then centring them by their new mean $\mu_N$. We justify the truncation when $\beta > 0$ is bounded from below by a function of $\alpha.$ 

\begin{lemma}[Truncated and Centered Matrix]
\label{lem:trunc}
Let $\epsilon > 0$ and consider $\beta = \frac{1}{4} + \frac{1}{\alpha} + \epsilon$. With $\X_N$ and $\BA_N$ as in Definition~\ref{def:hlfhvywish}, 
and $P$ and $N$ are growing with respect to a parameter $D$ as in Proposition~\ref{prop:mplaw}, define the matrix $\hat\X_N$ whose entries are
\begin{align*}
\hat{x}_{i,j} &:= x_{i,j} \mathbf{1}_{|x_{i,j}| < N^\beta} - \mu_N, \\
\mu_N&:= \Exp\big[x_{i,j} \mathbf{1}_{|x_{i,j}|<N^\beta}\big].
\end{align*}
If we define $\hat\BA_N := \frac{\hat\X_N \hat \X_N^*}{N}$ and 
\[
 \hat\theta_N(z) := \frac{1}{N^{1-\frac{\alpha}{4}}}\big( \Tr \G_{\hat \BA_N}(z)  - \Exp[\Tr \G_{\hat \BA_N}(z)]\big),
\]
then 
\[
|\hat\theta_N(z) - \theta_N(z)| \to 0
\]
in probability as $D\to \infty$.
\end{lemma}

\begin{proof}
From the matrix $\X_N$ we build the matrix $\tilde\X_N$ whose entries are equal to
\begin{equation} 
\label{eqn:truncate}
\tilde{x}_{i,j} := x_{i,j} \mathbf{1}_{|x_{i,j}| < 
N^{\beta}}.
\end{equation}
For each $1\leq i\leq P$, set the random variable $\xi_i$ to equal 1 
if the $i$-th row of $\X_N$ differs from the $i$-th row of $\tilde\X_N$, and equal to 0 otherwise, i.e. 
\[
\xi_i := \mathbf{1}_{ \bigcup_{j=1}^N \{ x_{i,j} \neq \tilde{x}_{i,j} 
\}} = \mathbf{1}_{ \bigcup_{j=1}^N \big\{ |x_{i,j}| \geq N^{\beta} 
\big\}}.
\]
Since $R:=\sum_{i=1}^P \xi_i$ counts the total number of non-zero rows
of $\X_N-\tilde\X_N$, we have the following bound
\[
\rank(\X_N-\tilde\X_N) \leq R.
\]
The rank of $\Exp[\tilde\X_N]$ is at most 1, 
which means $\rank(\tilde\X_N - \hat\X_N) \leq 1$.
Using Corollary~\ref{cor:resolvrank} we have the bound
\begin{equation}
\label{eqn:restrunc}
\frac{1}{P^{1-\frac{\alpha}{4}}}\big|\Tr \G_{\hat\BA_N}(z) - 
\Tr \G_{\BA_N}(z)\big| \leq \frac{\pi \left(R+1\right)}{|\Im z| 
P^{1-\frac{\alpha}{4}}}.
\end{equation}

To prove the Lemma we combine bound \eqref{eqn:restrunc} with the proof of convergence of  $\frac{R}{P^{1-\frac{\alpha}{4}}}$ to 0 in probability. Firstly, we will estimate $\Exp R$ and $\Var R.$ Afterwards, we will apply Chebychev’s
inequality.  

Variables $\xi_i$ are i.i.d since the rows of $\mathbf{X}_N$ are i.i.d., which yields
$\Exp R  = P ~\Exp \xi_1$ 
and $\Var R  = P ~ \Var \xi_1.$ By the 
distributional assumption of Definition~\ref{def:heavytails}, there 
exists a constant $C>0$ depending only on $\alpha$, $\beta$ and $\ell$ 
such that for sufficiently large $N$ holds
\[
\Prob\Big( \big\{|x_{1,1}| \geq N^{\beta}\big\} \Big) \leq 
\frac{C}{N^{\alpha\beta}}.
\]
Therefore for sufficiently large $N,$ the following bound holds
\begin{align*}
\Exp[\xi_i] &= 1 - \bigg( 1 - \Prob\Big(\big\{|x_{1,1}| \geq  N^{\beta} 
\big\} \Big)\bigg)^N \\
&\leq 1 - \exp\bigg( N \log\Big[1 - \frac{C}{N^{\alpha\beta}}\Big]\bigg)
\leq 1 - \exp\big( -2C N^{1-\alpha\beta}\big),\\
&\leq 2 C N^{1-\alpha \beta}.
\end{align*}
In the second inequality we used that $C N^{-\alpha \beta} < 
\frac{1}{2}.$ In the third inequality we used that
$\beta > \frac{1}{\alpha}$ so that $CN^{1-\alpha\beta} < \frac{1}{2}.$ 
Since $\xi_i$ only takes values $0$ and $1,$ $\Var \xi_i^2 \leq \Exp 
[\xi_i^2] = \Exp[\xi_i.]$ Thus, the following bound also holds
\[
\Var \xi_i \leq 2 CN^{1-\alpha\beta}.
\]

The bounds on $\xi_i$ and the fact that $P/N \to y \in (0,\infty)$ yield the existence of a deterministic constant $\tilde{C} > 0$ for
which, when $P$ is sufficiently large,
\[
\Exp R \leq \tilde{C} P^{2-\alpha \beta} \quad\hbox{and}\quad 
\Var R \leq \tilde{C}P^{2-\alpha\beta}.
\] 
 We rewrite the upper bound of equation~\eqref{eqn:restrunc} as
\begin{equation}
\label{eqn:randrank}
\frac{R}{P^{1-\frac{\alpha}{4}}} = \frac{R-\Exp 
R}{P^{1-\frac{\alpha}{4}}} + \frac{\Exp R}{P^{1-\frac{\alpha}{4}}}.
\end{equation}
Assume further that  $P$ is sufficiently large. The second term in equation~\eqref{eqn:randrank}
\[
\frac{\Exp R}{P^{1-\frac{\alpha}{4}}} \leq \tilde{C} P^{1 - \alpha \beta + \frac{\alpha}{4}} \to 0
\]
since $\beta >\frac{1}{\alpha} + \frac{1}{4} $.
By Chebychev's inequality, for any $t>0$
\[
\Prob\Big(\big\{|R - \Exp R| \geq t P^{1-\frac{\alpha}{4}}\big\}\Big) \leq \frac{\Var R}{t^2 P^{2-\frac{\alpha}{2}}}\leq
\frac{\tilde{C} P^{\alpha( \frac{1}{2} - \beta)}}{t^2} \to 0,
\]
(since $\beta > \frac{1}{\alpha} + \frac{1}{4}$ and $2 < 
\alpha < 4$ it follows that $
\Big(\frac{1}{2} - \beta \Big)\alpha < 0$). 
Thus, the first term in equation~\eqref{eqn:randrank} converges in 
probability to $0$. The sum of both terms converge to $0$ in probability, which in combination with inequality \eqref{eqn:restrunc} yields the statement of the Lemma.
\end{proof}

\subsection{Properties of truncated elements}\label{sec:trunc_prop}

We will use of the following properties of the entries of 
$\hat{\X}_N$, analogous to those proven in 
\cite[Lemma 3.1]{BGM16}.
\begin{lemma}[Truncation Bounds]
\label{lem:truncbound}
With the same notation as in Lemma~\ref{lem:trunc}, the following 
bounds hold for some $\kappa > 0$ depending only on $\beta$, $\alpha$
and $\ell$.
\begin{enumerate}
\item $|\mu_N|\leq \kappa N^{\beta - \alpha}$.
\label{trunc:meanbound}
\item Defining $\sigma_N^2 := \Exp|\hat{x}_{i,j}|^2$, then we have 
$|\sigma_N^2 -1|\leq \kappa N^{2\beta(1-\alpha)}$.
\label{trunc:varbound}
\item $\Exp|\hat{x}_{i,j}|^3\leq \kappa N^{\beta(3-\alpha)_+}$   
\label{trunc:3rdmom}
\item $\Exp|\hat{x}_{i,j}|^4\leq \kappa N^{\beta(4-\alpha)}$
\label{trunc:4thmom}
\item 
\label{trunc:charfunc}
For any $\lambda\in \C$ such that $\Im\lambda\leq 0$,
\begin{equation}
\label{characteristicfunction}
\begin{split}
\phi_N(\lambda)&:= 
\Exp\bigg[\exp\bigg(-\iunit\left|\frac{\hat{x}_{i,j}}{\sqrt{N}}\right|^2\lambda\bigg)\bigg],\\
&=1-\frac{\iunit\lambda\sigma_N^2}{N}+c\frac{\left(\iunit\lambda\sigma_N^2\right)^{\alpha/2}}{N^{\alpha/2}}+\frac{|\lambda\sigma_N|^{\alpha/2}}{N^{\alpha/2}}\mathfrak{E}_N(\mathbf{i}\lambda\sigma_N^2/N),
\end{split}
\end{equation}
where the function $\mathfrak{E}_N(z)$ is analytic in $z$ on $\Re(z)>0,$ uniformly on $N$ bounded on any compact there and uniformly on $N$ tends to zero when $z\to 0.$ 
\end{enumerate}
\end{lemma}
\begin{proof}
We use Definition~\ref{def:heavytails} to obtain the bounds 
on the integrals and prove parts 1---4 of the Lemma.
Since $x_{i,j}$ is centred,
\[
|\mu_N| = \Big|\Exp\big[x_{i,j} \mathbf{1}_{|x_{i,j}|\geq 
N^\beta}\big]\Big| \leq \int_{N^\beta}^\infty 
\Prob\big(\big\{|x_{i,j}|>t\big\}\big)\diff t \leq  \kappa 
N^{\beta(1-\alpha)},
\]
which proves the first part of the Lemma. Since 
the variance of $x_{i,j}$ is $1$,
\[
\big|\sigma_N^2 - 1\big| = \Exp\big[x_{i,j}^2 \mathbf{1}_{|x_{i,j}|\geq 
N^\beta}\big] =  \int_{N^\beta} t\Prob\big(\big\{ |x_{i,j}| > 
t\big\}\big) \,\diff t \leq \kappa N^{2\beta(1-\alpha)},
\]
which concludes the second part. We obtain analogously the third and the fourth part. 

It is left to prove the part~\ref{trunc:charfunc}. Let $x$ have the same 
distribution as $x_{i,j}$. We define
\[
\phi_N^{(1)}(\lambda) := \Exp\bigg[ \exp\bigg(- \frac{ \iunit \lambda|x 
- \mu_N|^2 }{N} \bigg)\bigg].
\]
Next, we estimate the difference
\begin{multline}
\Delta_N^{(1)}(\lambda):=\phi_N^{(1)}(\lambda) - \phi_N(\lambda) \\
= \Exp\bigg[ \bigg\{\exp\bigg( -\frac{ \iunit \lambda|x - \mu_N|^2}{ 
N} \bigg) - \exp\bigg( -\frac{\iunit \lambda 
|\mu_N|^2}{N}\bigg)\bigg\}\mathbf{1}_{|x|> N^\beta}\bigg].
\end{multline}
Using that for any $t \in\R$, since $\Im(\lambda)\leq 0$,
$|\exp(-\iunit\lambda t)| \leq 1,$ we obtain the bound
\begin{equation}
\sup_{\lambda\in \C^-}|\Delta_N^{(1)}(\lambda)| \leq 2\Prob\big(\{|x|>
N^\beta\}\big) \leq \kappa N^{-\alpha \beta}. 
\end{equation}
Next, we consider the function
\[
\phi^{(2)}_N(\lambda) := \Exp\bigg[ \exp\bigg( -\frac{ \iunit \lambda 
|x|^2 }{ N} \bigg)\bigg],
\]
and estimate the difference
\begin{multline}
\Delta_N^{(2)}(\lambda) :=\phi^{(1)}_N(\lambda) - 
\phi^{(2)}_N(\lambda) \\ = \Exp\bigg[ 
\bigg\{\exp\bigg(-\frac{\iunit\lambda|x-\mu_N|^2}{N}\bigg) - 
\exp\bigg(-\frac{\iunit \lambda |x|^2 }{N} \bigg)  \bigg\}\bigg].
\end{multline}
Using the 
representation
\[
-\iunit \lambda \int_{p_1}^{p_2} \exp(-\iunit\lambda s)\,\diff s = \exp(-\iunit\lambda p_1) - 
\exp(-\iunit\lambda p_2), \qquad p_1,p_2 \in \R^+,
\]
we have the bound
\[
\bigg|\exp\bigg(-\frac{\iunit\lambda|x-\mu_N|^2}{N}\bigg) - 
\exp\bigg(-\frac{\iunit \lambda |x|^2 }{N} \bigg)  \bigg| \leq
\min\bigg( 2 , \frac{|\lambda|\big||x-\mu_N|^2 - |x|^2\big|}{N}\bigg).
\]
Hence,
\begin{multline}
\big|\Delta_N^{(2)}(\lambda) \big| \leq 2 
\Prob\Big(\big\{ |\lambda|\big||x-\mu_N|^2 - |x|^2\big| > 2N 
\big\}\Big)+ \frac{|\lambda|}{N} \Exp\Big[\big| |x-\mu_N|^2 - 
|x|^2\big|\Big]\\ \leq\frac{3|\lambda|}{N} \Exp\Big[\big| |x-\mu_N|^2 - 
|x|^2\big|\Big] = \frac{3|\lambda|}{N} o(1) .
\end{multline}

Now we expand $\phi_N^{(2)}(\lambda)$ using the same proof of 
\cite[Theorem 8.1.6]{bingham_regular_1989} (where $\lambda \in \R$) adapted for complex $\lambda$. We
do the expansion for $\Im z \leq 0$ and $|z| \to 0$, for the function
\[
\psi(z) = \Exp\big[ \exp(-iz|x|^2)\big]
\]
and will apply this expansion to $z = \frac{\lambda}{N}$ which makes
$\psi(z) = \phi_N^{(2)}(\lambda)$.

Let $\tilde{F}$ be the distribution function of $|x|^2$ and let $G = 1 
- \tilde{F}.$ By Definition~\ref{def:heavytails} we have
\[
G(u) = \Prob\big( \big\{ |x|^2 > u \big\}\big)\sim \frac{c\ell(\sqrt{u})
u^{-{\frac{\alpha}{2}}} }{ -\Gamma\left(1-\frac{\alpha}{2}\right)} \qquad u \to \infty.
\]
%what is a Stieltjes function
We may write
\[
\psi(z) = \int_0^\infty \exp\big[-\iunit z u \big] \diff \tilde{F}(u)
= - \int_0^\infty \exp \big[-\iunit z u \big] \diff G(u),
\]
so that
\[
1 - \psi(z) = \int_0^\infty \big( \exp[-\iunit z u] - 1) \diff G(u),
\]
We integrate by parts this representation:
\[
1 - \psi(z) = \iunit z \int_0^\infty \exp[- \iunit z u] G(u)\,\diff u.
\]
Recall that 
\[
\int_0^\infty G(u)\,\diff u = \Exp\big[|x|^2\big] = 1.
\]
Combining the equations above, we get
\[
1 -\iunit z - \psi(z) = \iunit z \int_0^\infty \big(\exp[- \iunit z u] 
- 1\big) G(u)\,\diff u.
\]
Denote
$\omega :=\frac{z}{|z|}= \exp[\iunit \arg(z)]$. We change variables $u\mapsto t /|z|$,
so that 

\[
1 -\iunit z - \psi(z) = \iunit \omega \int_0^\infty \big(\exp[- \iunit 
\omega t] - 1\big) G\bigg(\frac{t}{|z|}\bigg)\,\diff t,
\]
which implies
\[
\frac{1 -\iunit z - \psi(z)}{G\big(\frac{1}{|z|}\big)} = \iunit \omega 
\int_0^\infty \big(\exp[- \iunit \omega t] - 1\big) 
\frac{G\big(\frac{t}{|z|}\big)}{G\big(\frac{1}{|z|}\big)}\,\diff t.
\]
For $|z| \to 0$, the integral on the right-hand side
converges uniformly by $\omega$ to
% here you need to justify something that it is uniform over w or so 
\begin{equation}
\iunit \omega \int_0^\infty \frac{\big(\exp[- \iunit \omega t] - 
1\big)}{t^{\frac{\alpha}{2}}} \,\diff t.
\end{equation}
We evaluate this integral using the same method as in \cite[Lemma 
11]{LM22}
\begin{multline}
\iunit \omega \int_0^\infty \frac{\big(\exp[- \iunit \omega t] - 
1\big)}{t^{\frac{\alpha}{2}}} \,\diff t = \iunit \omega 
\int_0^\infty t^{-\frac{\alpha}{2}} \int_0^1 \frac{\diff}{\diff 
\nu} \exp(-\iunit \omega \nu t) \,\diff \nu \,\diff t,\\
=-(\iunit \omega)^2
\int_0^1 \int_0^\infty  t^{1-\frac{\alpha}{2}} \exp(-\iunit \omega \nu t) \,\diff t \,\diff \nu,\\
=-(\iunit \omega)^2 \int_0^1 \nu^{\frac{\alpha}{2}-2}\,\diff \nu 
\int_0^\infty s^{1-\frac{\alpha}{2}} \exp(-\iunit \omega s) \,\diff 
s ,\\
= -(\iunit \omega)^2 \frac{\Gamma\big(\frac{\alpha}{2} - 1 
\big)}{\Gamma\big(\frac{\alpha}{2}\big)} \int_0^\infty 
s^{1-\frac{\alpha}{2}}\exp(-\iunit \omega s)\,\diff s\\
= -(\iunit \omega)^2 \frac{\Gamma\big(\frac{\alpha}{2} - 1 
\big)\Gamma\big(2-\frac{\alpha}{2}\big)}{\Gamma\big(\frac{\alpha}{2}\big)} \big(-\iunit \overline{\omega}\big)^{1- \frac{\alpha}{2}} \big(-\iunit \overline{\omega}\big)\\
= -\big(\iunit \omega\big)^{\frac{\alpha}{2}} 
	\frac{\Gamma\big(\frac{\alpha}{2} - 1\big)\Gamma\big(2-\frac{\alpha}{2}\big)}{\Gamma\big(\frac{\alpha}{2}\big)} = (\iunit \omega)^{\frac{\alpha}{2}} \Gamma\bigg(1-\frac{\alpha}{2}\bigg).
\end{multline}
Recalling, that 
\[
	G\left(\frac{1}{|z|}\right)\sim-\frac{c}{\Gamma\left(1-\frac{\alpha}{2}\right)}|z|^{\frac{\alpha}{2}}
\]
we  conclude, that 
\[
\psi(z)=1-\iunit z +c \left(\iunit z\right)^{\alpha/2}+|z|^{\alpha/2}o(1).
\]
\end{proof}

In the next sections, to shorten the notations, we will write $\X_N$ instead of $\hat{\X}_N$ and $x_{i, j}$ instead of $\hat{x}_{i, j}.$
\subsection{Martingale decomposition}\label{sec:mart_decomposition}
To prove Theorem \ref{thm:main} we namely need to show that for any $z_1, z_2, \dots, z_k \in \mathbb{C}\backslash\mathbb{R}$ a vector $\left<\theta_N(z_1), \theta_N(z_2), \dots, \theta_N(z_k)\right>$ converges to a complex Gaussian vector in distribution with the proper covariance matrix. Using, that $\theta_N(\overline{z}) = \overline{\theta_N(z)}$ it is equal to the fact, that for any $k$ for any $z_1, z_2, \dots 
 z_k\in \C\backslash \R$ and $\alpha^{\Re}_1, \alpha_2^{\Re}, \dots \alpha_k^{\Re}, \alpha^{\Im}_1, \alpha_2^{\Im}, \dots \alpha_k^{\Im}\in \R$ the linear combination 
\begin{multline}
    \alpha^{\Re}_1\cdot\left(\theta_N(z_1)+\theta_N(\overline{z}_1)\right)+ \dots +  \alpha^{\Re}_k\cdot\left(\theta_N(z_k)+\theta_N(\overline{z}_k)\right)\\
    +\iunit\left(\alpha^{\Im}_1\cdot\left(\theta_N(z_1)-\theta_N(\overline{z}_1)\right)+ \dots +  \alpha^{\Im}_k\cdot\left(\theta_N(z_k)-\theta_N(\overline{z}_k)\right)\right) \label{eqn:linear_comb}
\end{multline}
converges in distribution to a Gaussian random variable whose variance agrees to the covariance kernel form Theorem \ref{thm:main}.
Define the filtration $\cF_{N,k}$ where $0 \leq k \leq N$ to be the $\sigma$-algebra generated
by the first $k$-columns of the matrix $\X_N$. Using this filtration, we will
apply the Martingale Central Limit Theorem (Lemma \ref{lemma:martingale_clt}) to the decomposition into the martingale difference sequence of the random variable \eqref{eqn:linear_comb}.

 To shorten the notations, we denote $\mathbb{E}_k:=\mathbb{E}\left[\cdot|\mathcal{F}_{N, k}\right].$
Consider the following array:
\begin{equation}
Y_k(z):=\frac{1}{N^{1-\alpha/4}}\left(\mathbb{E}_{k}-\mathbb{E}_{k-1}\right)\left[\Tr \frac{1}{z-\mathbf{A}_N}\right].
\end{equation}
The independence of $\mathbf{x}_k$ and $\mathbf{A}_{N, k}:= \mathbf{A}_N -\frac{1}{N}\mathbf{x}_k \mathbf{x}_k^*$ yields

$$\left(\mathbb{E}_{k}-\mathbb{E}_{k-1}\right)\left[\Tr \frac{1}{z-\mathbf{A}_{N, k}}\right]=0.$$ 
This way, we rewrite
\begin{multline}
Y_k(z)=\frac{1}{N^{1-\alpha/4}}\left(\mathbb{E}_{k}-\mathbb{E}_{k-1}\right)\left(\Tr \frac{1}{z-\mathbf{A}_N}\right)= \\
\frac{1}{N^{1-\alpha/4}}\left(\mathbb{E}_{k}-\mathbb{E}_{k-1}\right)\left(\Tr\left[ \frac{1}{z-\mathbf{A}_N}-\frac{1}{z-\mathbf{A}_{N, k}}\right]\right). 
\label{eqn:yk}
\end{multline}
Corollary \ref{lem:wishartdiff} yields that
\begin{equation}
    Y_k(z) \leq \frac{\pi}{N^{1-\alpha/4}|\Im z|}.
\end{equation}
Since the expression \eqref{eqn:linear_comb} is the linear combination of $Y_k(z)$ for different $z,$ the equation above provides us with the first condition of Lemma \ref{lemma:martingale_clt}.
We conclude, that Theorem \ref{thm:main} follows from the lemma below. 
 \begin{lemma}
 For all pairs $z, w \in \C\backslash \R$ 
 \begin{equation}
     \sum_{k=1}^N \Exp_{k-1}\left[ Y_k(z)Y_k(w)\right] \overset{\Prob}{\rightarrow} C(z, w).
 \end{equation} \label{lemma:cltcovneeded} 
 \end{lemma}
 
Denote $\mathbf{G}_{N, k}(z):=\left(z-\mathbf{A}_{N, k}\right)^{-1}.$ We rewrite the right-hand side of equation \eqref{eqn:yk} as in \cite[p.54]{BS10}:  
\begin{equation}
Y_k(z) = \frac{1}{N^{1-\alpha/4}} \left(\mathbb{E}_{k}-\mathbb{E}_{k-1}\right) \frac{ \frac{1}{N}\mathbf{x}_k^* \mathbf{G}_{N, k}(z) ^{2} \mathbf{x}_k}{1- \frac{1}{N}\mathbf{x}_k^* \mathbf{G}_{N, k}(z) \mathbf{x}_k}.
\label{martingale}
\end{equation}

From the equation above, 

\begin{multline}
Y_k(z) = \frac{1}{N^{1-\alpha/4}} \left(\mathbb{E}_{k}-\mathbb{E}_{k-1}\right)\left[ \frac{ -\frac{1}{N}\mathbf{x}_k^* \mathbf{G}_{N, k}(z) ^{2} \mathbf{x}_k}{1- \frac{1}{N}\mathbf{x}_k^* \mathbf{G}_{N, k}(z) \mathbf{x}_k}+\frac{1}{z}\right] \\ = \frac{1}{N^{1-\alpha/4}} \left(\mathbb{E}_{k}-\mathbb{E}_{k-1}\right) \frac{\frac{\partial}{\partial z} \left(z- \frac{1}{N}\mathbf{x}_k^* z \mathbf{G}_{N, k}(z) \mathbf{x}_k\right)}{z- \frac{1}{N}\mathbf{x}_k^* z \mathbf{G}_{N, k}(z) \mathbf{x}_k}.
\label{eqn:Y}
\end{multline}

For short,  denote $g_{N, k}(z):= z- \frac{1}{N}\mathbf{x}_k^* z \mathbf{G}_{N, k}(z) \mathbf{x}_k.$ Then, using this notation we futher rewrite \eqref{eqn:Y} as

\begin{multline}
Y_k(z) = \frac{1}{N^{1-\alpha/4}} \left(\mathbb{E}_{k}-\mathbb{E}_{k-1}\right) \frac{\frac{\partial}{\partial z} g_{N, k}(z)}{g_{N, k}(z)}\\=\frac{1}{N^{1-\alpha/4}} \left(\mathbb{E}_{k}-\mathbb{E}_{k-1}\right) \frac{\partial}{\partial z}\log\left|g_{N, k}(z)\right|^2.
\end{multline}
We swap the derivative and the expectation in the equation above to get
\begin{equation}
Y_k(z)= \frac{\partial}{\partial z} \frac{1}{N^{1-\alpha/4}} \left(\mathbb{E}_{k}-\mathbb{E}_{k-1}\right) \log\left|g_{N, k}(z)\right|^2,
\label{eqn:derswapg}
\end{equation}
and justify the swap the following way.
Application of Lemma \ref{lemma:diagest} and Lemma \ref{lemma:imsgn} to the definition of $g_{N, k}(z)$ yields
\begin{equation}
|\Im z| \leq |{g}_{N, k}(z)| \leq |z|+ \frac{|z|}{|\Im z|} \frac{1}{N}\|\mathbf{x}_k\|^2_2. 
\label{eqn:gest}
\end{equation} 
Thus, 
\begin{equation}
2\log|\Im z| \leq \log\left|g_{N, k}(z)\right|^2\leq 2\log\left( |z|+ \frac{|z|}{|\Im z|} \frac{1}{N}\|\mathbf{x}_k\|^2_2\right).
\end{equation}
The function $\log\left|g_{N, k}(z)\right|^2$ is harmonic as it is a real part of the holomorphic function. Thus, a combination of Lemma \ref{lemma:poissonkernel} and Lemma \ref{lemma:derivative_swap} justifies the swap. 

\subsection{Diagonalization}\label{sec:diagonalisation}

Further $\operatorname{diag}\left[\mathbf{G}_{N, k} \right]$ will denote a $P\times P$ matrix, whose diagonal entries match those of the matrix $\mathbf{G}_{N, k}$, and off-diagonal entries are equal to $0.$

Let us denote 

\begin{equation}
\Tilde{g}_{N, k}(z):=z- \frac{1}{N}\mathbf{x}_k^* z \operatorname{diag}\left[\mathbf{G}_{N, k}(z)\right] \mathbf{x}_k.
\end{equation}

Define
\begin{equation}
\tilde{Y}_k(z): = \frac{1}{N^{1-\alpha/4}} \left(\mathbb{E}_{k}-\mathbb{E}_{k-1}\right) \frac{\frac{\partial}{\partial z} \Tilde{g}_{N, k}(z)}{\Tilde{g}_{N, k}(z)}.
\label{eqn:ytilde}
\end{equation}

Similarly equation\eqref{eqn:derswapg}, equation \eqref{eqn:ytilde} can be rewritten as

\begin{equation}
\Tilde{Y}_k(z)= \frac{\partial}{\partial z} \frac{1}{N^{1-\alpha/4}} \left(\mathbb{E}_{k}-\mathbb{E}_{k-1}\right) \log\left|\Tilde{g}_{N, k}(z)\right|^2.
\end{equation}

The main purpose of the further computations in this section will be to justify the replacement of $Y$ with $\tilde{Y}$ in Lemma \ref{lemma:cltcovneeded}.
    
\begin{lemma}
\label{lemma:diagonalisation}
For $0<t<y$ and $\Tilde{Y}_k$ defined as above the following there exist $C, \delta>0$ such that:
$$
\left|\sum_{k=1}^{N} \mathbb{E}_{k-1}\left[Y_{k}(z) Y_{k}(w)-\tilde{Y}_{k}(z) \tilde{Y}_{k}(w)\right]\right|\leq C N^{-\delta}\frac{|z||w|}{|\Im z|^3|\Im w|^3}.
$$
\end{lemma}

Define the operator $\Cov_{k-1}[\cdot, \cdot]$ by 
\begin{multline}
    \Cov_{k-1}[a, b]:=\Exp_{k-1}\Bigl[\left(\Exp_k[a]-\Exp_{k-1}[a]\right)\left(\Exp_k[b]-\Exp_{k-1}[b]\right)\Bigr]\\=\Exp_{k-1}\bigl[\Exp_k\left[a\right]\Exp_k\left[b\right]\bigr]-\Exp_{k-1}\bigl[a\bigr]\Exp_{k-1}\bigl[b\bigr].
\end{multline}
This notation will simplify some computations because for random variable $\xi,$ that is independent of $\x_k$
\begin{equation}
    \Cov_{k-1}[a+\xi, b] = \Cov_{k-1}[a, b].
\end{equation}
Denote as $F_k(z):=\log|g_{N, k}(z)|^2$ and $\Tilde{F}_k(z):=\log |\Tilde{g}_{N, k}(z)|^2.$ 
As it was done previously to justify \eqref{eqn:derswapg}, we justify the swap of derivative and expectation once again to get
\begin{equation}
    \mathbb{E}_{k-1}\left[Y_{k}(z) Y_{k}(w)\right]=N^{\alpha/2-2}\frac{\partial}{\partial z} \frac{\partial}{\partial w} \Cov_{k-1}\left[F_k(z), F_k(w)\right]
\end{equation}
and
\begin{equation}
    \mathbb{E}_{k-1}\left[\tilde{Y}_{k}(z) \tilde{Y}_{k}(w)\right]=N^{\alpha/2-2}\frac{\partial}{\partial z} \frac{\partial}{\partial w}\Cov_{k-1}\left[\tilde{F}_k(z), \tilde{F}_k(w)\right].
\end{equation}
Using Corollary \ref{col:derest} is is easy to see that statement of the Lemma \ref{lemma:diagonalisation} would follow from the existence of $C, \delta>0,$ independent of $k$ and $N$ such that

\begin{multline}
N^{\alpha/2-2}\bigg|\Cov_{k-1}\Bigl[F_k(z), F_k(w)\Bigr]-\Cov_{k-1}\Bigl[\Tilde{F}_k(z),\Tilde{F}_k(w)\Bigr]\bigg|\\ \leq \frac{CN^{-\delta}}{N}\frac{|z||w|}{|\Im z|^2|\Im w|^2}.
\end{multline}

Define also functions $\hat{g}_{N, k}(z):=z-\frac{1}{N}z\Tr \mathbf{G}_{N, k}(z)$ and $\Hat{F}_k(z):=\log|\Hat{g}_{N, k}(z)|^2.$ Notice, that they are independent of the $k$-th column. Denote $\epsilon_k(z):=F_k(z)-\tilde{F_k}(z)$ and $\psi_k(z):=\Tilde{F}_k(z)-\Hat{F}_k(z).$ Using that $\Exp_k\Hat{F}_k(z)\in \mathcal{F}_{N, k}$  
we see, that
\begin{multline}
    \Cov_{k-1}\Bigl[F_k(z), F_k(w)\Bigr] = \Cov_{k-1}\Bigl[\Hat{F}(z)+\psi_k(z)+\epsilon_k(z), \Hat{F}(w)+\psi_k(w)+\epsilon_k(w)\Bigr]\\ =\Cov_{k-1}\Bigl[\psi_k(z)+\epsilon_k(z), \psi_k(w)+\epsilon_k(w)\Bigr]
\end{multline}
and
\begin{multline}
\Cov_{k-1}\left[\Tilde{F}_k(z),\Tilde{F}_k(w)\right] = \Cov_{k-1}\left[\Hat{F}(z)+\psi_k(z), \Hat{F}(w)+\psi_k(w)\right]\\= \Cov_{k-1}\Bigl[\psi_k(z), \psi_k(w)\Bigr].
\end{multline}
Thus
\begin{multline}
    \Cov_{k-1}\left[F_k(z), F_k(w)\right]-\Cov_{k-1}\left[\Tilde{F}_k(z),\Tilde{F}_k(w)\right]\\=\Cov_{k-1}\left[\psi_k(z) \epsilon_k(w)\right]+\Cov_{k-1}\left[\epsilon_k(z), \psi_k(w)\right]+\Cov_{k-1}\left[\epsilon_k(z), \epsilon_k(w)\right]
    \\=: T_1+T_2+T_3.
\end{multline}
This way, Lemma \ref{lemma:diagonalisation} follows from the lemma below: 
\begin{lemma}
Consider $T_1, T_2, T_3$ defined above. 
There exists $C, \delta>0,$ independent of $k, N$ such that
$$N^{\alpha/2-1}|T_i|<O(1) N^{-\delta}\frac{|z||w|}{|\Im z|^2 |\Im w|^2}$$
for all $z, w\in \C \backslash \R$ and all $i\in\left\{1, 2, 3\right\}.$
\label{lemma:T}
\end{lemma}

\begin{proof}

Denote:
\begin{equation}
    \eta_k(z):=g_{N, k}(z) - \Tilde{g}_{N, k}(z)=\frac{1}{N} z \sum_{i\neq j}\overline{x_{ik}}x_{jk}\left(\mathbf{G}_{N, k}(z)\right)_{ij}
    \label{eqn:eta}
\end{equation}

and 

\begin{equation}E_k(z):=\Tilde{g}_{N, k}(z) - \hat{g}_{N, k}(z)=\frac{1}{N}z\sum_{i=1}^P \left(\mathbf{G}_{N, k}(z)\right)_{ii}\left(|x_{ik}|^2-1\right).\end{equation}

\begin{equation}
\epsilon_k(z)=\log\left|\frac{g_{N, k}(z)}{\Tilde{g}_{N, k}(z)}\right|^2 
= 2 \log\left|1+\frac{\eta_k(z)}{\tilde{g}_{N, k}(z)}\right| = 
-2 \log\left|1-\frac{\eta_k(z)}{g_{N, k}(z)}\right|
\label{eqn:eps1}
\end{equation}

Using that for $z\in \C$ $|1+z|\leq 1+|z|$ and that for $x>0$ holds $\log(1+x)\leq x,$ and $|\tilde{g}_{N, k}(z)|\geq |\Im z|$ 

\begin{equation}
\log\left|1+\frac{\eta_k(z)}{\Tilde{g}_{N, k}(z)}\right|
 \leq \log\left[ 1+\left| \frac{\eta_k(z)}{\Tilde{g}_{N, k}(z)}\right|\right]
\leq \frac{|\eta_k(z)|}{|\Im z|},
\label{eqn:eps2}
\end{equation}

and similarly, using $|g_{N, k}(z)|\geq |\Im z|$,
\begin{equation}
\log\left|1-\frac{\eta_k(z)}{g_{N, k}(z)}\right|
\leq \frac{|\eta_k(z)|}{|\Im z|}.
\label{eqn:eps3}
\end{equation}

Equation \eqref{eqn:eps1} and inequalities \eqref{eqn:eps2}, \eqref{eqn:eps3} allow us to conclude, that 
\begin{equation}
-2 \frac{|\eta_k(z)|}{|\Im z|} \leq \epsilon_k(z) \leq 2 \frac{|\eta_k(z)|}{|\Im z|}.\label{eqn:epsest}
\end{equation}

Using the same method, one also can prove that 
 \begin{equation}
-2 \frac{|E_k(z)|}{|\Im z|} \leq \psi_k(z) \leq 2 \frac{|E_k(z)|}{|\Im z|}.\label{eqn:thetaest}
\end{equation}
Combining \eqref{eqn:epsest} and \eqref{eqn:thetaest} with Cauchy--Schwarz inequality, we get
\begin{align}
|T_1| \leq 8\frac{\sqrt{\mathbb{E}_{k-1}|E_k(z)|^2\mathbb{E}_{k-1}|\eta_k(w)|^2}}{|\Im z| |\Im w|}\label{eqn:t1}\\
|T_2| \leq 8\frac{\sqrt{\mathbb{E}_{k-1}|\eta_k(z)|^2\mathbb{E}_{k-1}|E_k(w)|^2}}{|\Im z| |\Im w|}\label{eqn:t2}\\
|T_3| \leq 8\frac{\sqrt{\mathbb{E}_{k-1}|\eta_k(z)|^2\mathbb{E}_{k-1}|\eta_k(w)|^2}}{|\Im z| |\Im w|}.
\label{eqn:t3}
\end{align}
Now, we need to estimate $\mathbb{E}_{k-1}|\eta_k|^2$ and $\mathbb{E}_{k-1}|E_k|^2$.
\begin{lemma}
\label{lemma:eta}
\begin{enumerate}
For $\eta_k(z)$ defined above the following holds:
    \item $\mathbb{E}_{k-1}|\eta_k(z)|^2\leq O(1) N^{-1}\frac{|z|^2}{|\Im z|^2}.$
    \item
    $\mathbb{E}_{k-1}|E_k(z)|^2 \leq O(1)\frac{|z|^2}{|\operatorname{Im}z|^2}N^{4/\alpha - \alpha/4 -1 +\epsilon_0 },$ where $\epsilon_0 := \epsilon(4-\alpha)$
\end{enumerate}

\end{lemma}
\begin{proof}
The first part of the lemma can be proven by combining Ward identity (Lemma \ref{lemma:wardidentity}) and Lemma \ref{lemma:diagest}:
\begin{multline}
\mathbb{E}_{k-1}|\eta_k(z)|^2= \frac{1}{N^2}|z|^2 \mathbb{E}_{\mathbf{x}_k}|\sum
_{i\neq j}\overline{x_{ik}}x_{jk}(\mathbf{G}_{N, k}(z))_{i, j}|^2\\ \leq 4 |z|^2 \sigma_N^4 \frac{1}{N^2} \sum_{i, j}|(\mathbf{G}_{N, k}(z))_{i, j}|^2
= 4 \sigma_N^4 N^{-2} |z|^2 \frac{1}{|\Im z|}\sum_{i=1}^{P}|\Im \left(\mathbf{G}_{N, k}(z)\right)_{ii}|\\ \leq 4 \sigma_N^4 N^{-2} |z|^2 \frac{1}{|\Im z|}\sum_{i=1}^{P}|\left(\mathbf{G}_{N, k}(z)\right)_{ii}| \leq O(1) N^{-1}\frac{|z|^2}{|\Im z|^2}.
\end{multline}
The second part of the lemma follows from Lemma \ref{lemma:diagest} and part 4 of Lemma \ref{lem:truncbound}. 
\end{proof}
Combining Lemma \ref{lemma:eta} with equations \eqref{eqn:t1}, \eqref{eqn:t2} and \eqref{eqn:t3} we can see, that
\begin{equation}
|T_1|, |T_2|< O(1)\frac{ N^{-1/2 + (2/\alpha - \alpha/8 - 1/2) +\epsilon_0 /2 }|z||w|}{|\Im z|^2 |\Im w|^2}
\end{equation}
and
\begin{equation}
|T_3|< O(1) \frac{N^{-1}|z||w|}{|\Im z|^2 |\Im w|^2}.
\end{equation}
For any $\alpha\in (2, 4)$ there exist small enough $\epsilon>0$ such that 
\begin{equation}
-1/2 + (2/\alpha - \alpha/8 - 1/2) +\epsilon_0 /2 < 1 - \alpha/2.  
\label{eqn:epsilontochoose}
\end{equation}
Also, when $\alpha\in(2, 4)$
 $$
 -1<1 - \alpha/2,
 $$
 that completes the proof of Lemma \ref{lemma:T}. 
\end{proof}
\subsection{Computation of the limit}\label{sec:computation_of_limit}

Recalling Lemma \ref{lemma:diagonalisation}, to prove Lemma \ref{lemma:cltcovneeded} we should compute the limit of
$$
    \sum_{k=1}^{N}\mathbb{E}_{k-1}\left[ \tilde{Y}_k(z)\tilde{Y}_k(w)\right].
$$

Namely, we should prove that for fixed $z$ and $w$

\begin{equation}
     \frac{1}{N^{2-\alpha/2}}
     \sum_{k=1}^{N}\Cov_{k-1}\left[\frac{\frac{\partial}{\partial z}\tilde{g}_{N, k}(z)}{\tilde{g}_{N, k}(z)}, \frac{\frac{\partial}{\partial w}\tilde{g}_{N, k}(w)}{\tilde{g}_{N, k}(w)}\right] \stackrel{\Prob}{\rightarrow} C(z, w).
     \label{eqn:toproveconvergence}
\end{equation}

\begin{definition}[Uniform convergence in probability]

We say that the sequence array of random variables $X^{(k)}_N$ uniformly on $k$ converges in probability to the constant $C$ when $N \to \infty$ if for all $\epsilon>0$

\begin{equation}
    \max_{k}\Prob \left( |X_N^{(k)}-C|>\epsilon \right)
\underset{N \rightarrow \infty}{\rightarrow}
0.
\end{equation}

\end{definition}
By Lemma \ref{lemma:convergenceofaverages}, the convergence in equation \eqref{eqn:toproveconvergence} will follow from the Lemma below.

\begin{lemma}
Fix $z, w.$
\[
N^{\alpha/2-1}\Cov_{k-1}\left[\frac{\frac{\partial}{\partial z}\tilde{g}_{N, k}(z)}{\tilde{g}_{N, k}(z)}, \frac{\frac{\partial}{\partial w}\tilde{g}_{N, k}(w)}{\tilde{g}_{N, k}(w)}\right]\underset{N \rightarrow \infty}{\rightarrow} C(z, w)
\]
uniformly on $k$ in probability. 
Also, there exists a constant $D(z, w)>0$ such that for all $k, N$

\[
\left|N^{\alpha/2-1}\Cov_{k-1}\left[\frac{\frac{\partial}{\partial z}\tilde{g}_{N, k}(z)}{\tilde{g}_{N, k}(z)}, \frac{\frac{\partial}{\partial w}\tilde{g}_{N, k}(w)}{\tilde{g}_{N, k}(w)}\right]\right|<D(z, w).
\]
\label{lemma:twoconditions}
\end{lemma}
Further, we will do some preparatory work to rewrite expressions in Lemma \ref{lemma:twoconditions}. 
\begin{definition}[$k$-independent copy]
Fix $k.$ Denote $\underline{\X_N}$ a $P\times N$ matrix, with the same distribution as $\X_N,$ whose first $k$ columns match those of matrix $\X_N,$ and others are independent of $\X_N.$ For any random variable $a:=F(\X_N)$ for some non-random function $F$ on the space of $P\times N$ matrices we will denote $\underline{a}:=F(\underline{\X_N}).$
\end{definition}
Let $\mathcal{F}_{\x_k}$ be the $\sigma$-algebra, generated by all the columns of matrix $\X_N,$ apart from the column $\x_k.$ We will use the notations $\Exp_{\x_k}[\cdot]$ and $\Cov_{\x_k}[\cdot, \cdot]$ that act on any random variables $a, b$ the following way: 
\begin{equation}
    \Exp_{\x_k}[a]:= \Exp\left[a\mid \mathcal{F}_{\x_k}\right]
\end{equation}
and
\begin{equation}
    \Cov_{\x_k}[a, b]:= \Exp_{\x_k}[ab]-\Exp_{\x_k}[a]\Exp_{\x_k}[b].
\end{equation}
We generalize \cite[Lemma~3.4]{BGM16} for Sample Covariance matrices. 
\begin{lemma}
Suppose that $a=F_1(\X_N)$ and $b=F_2(\X_N),$ for some non-random functions $F_1, F_2$ on the space of $P\times N$ matrices, and $\Var a < +\infty,$ $\Var b < +\infty.$ Then 
\begin{equation}
    \Cov_{k-1}\left[a, b\right]= \Exp_k\left(\Cov_{\x_k}\left[a, \underline{b}\right]\right).
\end{equation}
\label{lemma:tildelemma}
\end{lemma}
\begin{proof}
$a$ and $\underline{b}$ are independent when the first $k$ columns of the matrix $\mathbf{X}_N$ are fixed. Thus,
\begin{equation}
\Exp_k(a\underline{b})=\Exp_k(a)\Exp_k(\underline{b}).
\label{eqn:tildecov}
\end{equation}
Notice, that $\Exp_k b = \Exp_k \underline{b}, $ 
which further yields
\begin{equation}
\Exp_{k-1}\left(\Exp_k(a)\Exp_k(b)\right)=\Exp_{k-1}\left(\Exp_k(a\underline{b})\right)=\Exp_{k-1}\left(a\underline{b}\right)=\Exp_{k}\Exp_{\x_k}\left(a\underline{b}\right).
\label{eqn:tildefirst}
\end{equation}
Notice, that
$\underline{\Exp_{\x_k} b}=\Exp_{\x_k} \underline{b}.$
Similarly to equation \eqref{eqn:tildecov}
\begin{equation}
\Exp_k\left(\Exp_{\x_k} a \Exp_{\x_k} \underline{b}\right) = \Exp_{k-1} a \Exp_{k-1}\underline{b}.
\label{eqn:secondtilde}
\end{equation}
Combining \eqref{eqn:tildefirst} and \eqref{eqn:secondtilde} we conclude the statement of the Lemma.
\end{proof}
Using Lemma \ref{lemma:tildelemma} , we conclude that

\begin{equation}
\Cov_{k-1}\left[ \frac{ \frac{\partial}{\partial z} \tilde{g}_{N, k}(z)}{\tilde{g}_{N, k}(z)} , \frac{\frac{\partial}{\partial w}\tilde{g}_{N, k}(w)}{\tilde{g}_{N, k}(w)} \right]=\Exp_{k}\left[\Cov_{\mathbf{x}_k} \left[\frac{ \frac{\partial}{\partial z} \tilde{g}_{N, k}(z)}{\tilde{g}_{N, k}(z)} ,  \frac{\frac{\partial}{\partial w}\underline{\tilde{g}_{N, k}}(w)}{\underline{\tilde{g}_{N, k}}(w)}\right] \right].
\label{eqn:rewritingwithtildelemma}
\end{equation}

 By Lemma \ref{lemma:imsgn}, $\Re \left(\sgn \Im z \iunit \tilde{g}_{N, k}(z)\right) \leq - |\Im z|.$ Thus, we can rewrite

\begin{multline}
\frac{\frac{\partial}{\partial z}\tilde{g}_{N, k}(z)}{\tilde{g}_{N, k}(z)}=-\iunit \sgn \Im z \frac{\partial}{\partial z}\tilde{g}_{N, k}(z)\times \int_0^{\infty} \exp\left( \sgn \Im z \iunit t \tilde{g}_{N, k} (z) \right) d t \\
=-\iunit \sgn \Im z \int_0^{\infty} \frac{\partial}{\partial z} \frac{1}{\iunit t \sgn \Im z}\exp\left(\sgn \Im z \iunit t \tilde{g}_{N, k}(z)\right) d t.
\label{eqn:integralovert}
\end{multline}

Substituting \eqref{eqn:integralovert} into \eqref{eqn:rewritingwithtildelemma} gives:

\begin{multline}
N^{\alpha/2-1}\Cov_{\mathbf{x}_k} \left[\frac{ \frac{\partial}{\partial z} \tilde{g}_{N, k}(z)}{\tilde{g}_{N, k}(z)} ,  \frac{\frac{\partial}{\partial w}\underline{\tilde{g}_{N, k}}(w)}{\underline{\tilde{g}_{N, k}}(w)}\right] = \\ =\int_0^{\infty}\int_0^{\infty}  \frac{\partial}{\partial z} \frac{\partial}{\partial w}N^{\alpha/2-1}\frac{\Cov_{\mathbf{x}_k}\left[\exp\left(\sgn \Im z \iunit t \tilde{g}_{N, k}(z)\right), \exp\left(\sgn \Im w \iunit s \underline{\tilde{g}_{N, k}}(w)\right)\right]}{ts} \diff t \diff s.
\label{eqn:expand_as_integral}
\end{multline}
Denote
$$
\mathcal{L}_{N, k}(z, t, w, s):=N^{\alpha/2-1}\frac{\Cov_{\mathbf{x}_k}\left[\exp\left(\sgn \Im z \iunit t \tilde{g}_{N, k}(z)\right), \exp\left(\sgn \Im w  \iunit s \underline{\tilde{g}_{N, k}}(w)\right)\right]}{ts}.
$$
Equation \eqref{eqn:expand_as_integral} allows to rewrite the first condition of the Lemma \ref{lemma:twoconditions} as 
\begin{equation}
\Exp_{k}\int_0^{\infty}\int_0^{\infty} \frac{\partial}{\partial z} \frac{\partial}{\partial w}\mathcal{L}_{N, k}(z, t, w, s) \diff t \diff s \underset{N \rightarrow \infty}{\rightarrow} \int_0^{\infty}\int_0^{\infty} \frac{\partial}{\partial z} \frac{\partial}{\partial w}\mathcal{L}(z, t, w, s) \diff t \diff s
\label{eqn:convergence}
\end{equation}

uniformly on $k$ in probability, and the second condition as

\begin{equation}
\left|\Exp_k\left[\int_0^{\infty}\int_0^{\infty} \frac{\partial}{\partial z} \frac{\partial}{\partial w}\mathcal{L}_{N, k}(z, t, w, s) \diff t \diff s\right]\right|<D(z, w).
\label{eqn:bound}
\end{equation}
Firstly, we will prove the bound \eqref{eqn:bound}.  The intermediate lemma from this proof will allow us to cut the domain of integration in \eqref{eqn:convergence}. 
\subsubsection{Proof of the bound \eqref{eqn:bound} and simplification of  \eqref{eqn:convergence}.}
\begin{lemma}
\begin{equation}
\left|\frac{\partial}{\partial z} \frac{\partial}{\partial w}\mathcal{L}_{N, k}(z, t, w, s)  \diff t \diff s \right|\leq \mathcal{S}(z, t, w, s),
\end{equation}
where 
$$\mathcal{S}(z, t, w, s)= 32N^{\alpha/2-1}\frac{\exp(\frac{-t|\Im z|-s|\Im w|}{2})}{|\Im z||\Im w|}\times\min\left(\frac{1}{t}\frac{|w|}{|\Im w|}, \frac{1}{s}\frac{|z|}{|\Im z| } \right).$$
\end{lemma}

\begin{proof}
By definition of $\tilde{g}$
\begin{multline*}
\mathcal{L}_{N, k}(z, t, w, s)=
N^{\alpha/2-1} \exp(\iunit \sgn_{\Im z}tz+\iunit \sgn_{\Im w}sw)\times \\ \Cov_{\mathbf{x}_k}\left[\frac{\exp\left(\sgn_{\Im z}\iunit t f_{N, k}(z)\right)}{t}, \frac{\exp\left(\sgn_{\Im w} \iunit s \underline{f_{N, k}}(w)\right)}{s}\right], \end{multline*}
where $f_{N, k}(z):=- \frac{1}{N}\mathbf{x}_k^* z \operatorname{diag}\left[\mathbf{G}_{N, k}(z)\right] \mathbf{x}_k$ and $\underline{f_{N, k}}(z):=- \frac{1}{N}\mathbf{x}_k^* z \operatorname{diag}\left[\underline{\mathbf{G}_{N, k}}(z)\right] \mathbf{x}_k.$ 
By the properties of the $\Cov$ operator, the following holds:
\begin{multline}
\Cov_{\mathbf{x}_k}\left[\frac{\exp\left(\sgn_{\Im z}\iunit t f_{N, k}(z)\right)}{t}, \frac{\exp\left(\sgn_{\Im w} \iunit s \underline{f_{N, k}}(w)\right)}{s}\right]=\\
\Cov_{\mathbf{x}_k}\left[\frac{\exp\left(\sgn_{\Im z}\iunit t f_{N, k}(z)\right)}{t}, \frac{\exp\left(\sgn_{\Im w} \iunit s \underline{f_{N, k}}(w)\right)-1}{s}\right]=\\
\Cov_{\mathbf{x}_k}\left[\frac{\exp\left(\sgn_{\Im z}\iunit t f_{N, k}(z)\right)-1}{t}, \frac{\exp\left(\sgn_{\Im w} \iunit s \underline{f_{N, k}}(w)\right)}{s}\right].
\end{multline}
One can bound 
\begin{multline}
    \left|\frac{\exp\left(\sgn_{\Im z} \iunit t f_{N, k}(z)\right)-1}{t}\right|=\\ \left|\frac{1}{t}\int_0^t f_{N, k}(z) \exp\left(\sgn_{\Im z}\iunit u f_{N, k}(z)\right) \diff u\right| \leq |f_{N, k}(z)|,
\end{multline}
and analogously 

\[
\left|\frac{\exp\left(\sgn_{\Im w}\iunit s \underline{f_{N, k}}(w)\right)-1}{s}\right|\leq |\underline{f_{N, k}}(w)|.
\]

The second part of Lemma \ref{lemma:imsgn} yields that $|\exp\left(\sgn_{\Im z} \iunit t f_{N, k}(z)\right)|\leq 1$ and that $|\exp\left(\sgn_{\Im w} \iunit s f_{N, k}(w)\right)|\leq 1.$ Using  Lemma \ref{lemma:diagest} we also  estimate  $\Exp_{\x_k}|f_{N, k}(z)|\leq \frac{|z|}{|\Im z|}.$

Thus, 
\[
\left|\Cov_{\mathbf{x}_k}\left[\frac{\exp\left(\sgn_{\Im z}\iunit t f_{N, k}(z)\right)}{t}, \frac{\exp\left(\sgn_{\Im w} \iunit s \underline{f_{N, k}}(w)\right)}{s}\right]\right|\leq 2 \min\left(\frac{1}{t}\frac{|w|}{|\Im w|}, \frac{1}{s}\frac{|z|}{|\Im z| } \right).
\]

Using that $\Re(\sgn_{\Im z} \iunit t z)\leq -|\Im z|$ and
$\Re(\sgn_{\Im w} \iunit s w)\leq -|\Im w|,$ we can bound $\left|\exp(\iunit \sgn_{\Im z}tz+\iunit \sgn_{\Im w}sw)\right|$ so that the inequality above leads to
\begin{equation}
  |\mathcal{L}_{N, k}(z, t, w, s)|\leq 2 N^{\alpha/2-1}\exp(-t|\Im z|-s|\Im w|)\times\min\left(\frac{1}{t}\frac{|w|}{|\Im w|}, \frac{1}{s}\frac{|z|}{|\Im z| } \right).
\end{equation}
Function $\mathcal{L}_{N, k}(z, t, w, s)$ is analytic by $z$ and $w.$ Applying Cauchy inequality (Lemma~\ref{lemma:cauchy_inequality}), we get that 
\begin{equation}
  \left|\frac{\partial}{\partial z}\frac{\partial}{\partial w}\mathcal{L}_{N, k}(z, t, w, s)\right|\leq 32 N^{\alpha/2-1}\frac{\exp(\frac{-t|\Im z|-s|\Im w|}{2})}{|\Im z||\Im w|}\times\min\left(\frac{1}{t}\frac{|w|}{|\Im w|}, \frac{1}{s}\frac{|z|}{|\Im z| } \right).
\end{equation}
\end{proof}

\begin{corollary}
Suppose $0<\epsilon<2-\frac{\alpha}{2}.$ Denote 
\begin{equation}
    \mathcal{L}_{N, k}^0(z, t, w, s):=\mathcal{L}_{N, k}(z, t, w, s)\times \mathbf{1}_{\frac{t|z|}{\left|\Im z\right|}<N^{1-\epsilon}}\times \mathbf{1}_{\frac{s|w|}{\left|\Im w\right|}<N^{1-\epsilon}}. 
\end{equation}
Then 
\begin{multline}
    \int_0^{\infty}\int_0^{\infty} \left|\frac{\partial}{\partial z} \frac{\partial}{\partial w}\mathcal{L}_{N, k}(z, t, w, s) -\frac{\partial}{\partial z} \frac{\partial}{\partial w}\mathcal{L}^0_{N, k}(z, t, w, s)\right|\diff t \diff s \\ \leq O(1)N^{\alpha/2 - 2+\epsilon} \frac{|z||w|}{|\Im z|^3 |\Im w|^3}.
\end{multline}
\end{corollary}
\begin{proof}
    Suppose that $t_0 = \frac{N^{1-\epsilon}\left|\Im z\right|}{|z|}$ and $s_0 = \frac{N^{1-\epsilon}\left|\Im w\right|}{|w|}.$
    Then $\min\left(\frac{1}{t}\frac{|w|}{|\Im w|}, \frac{1}{s}\frac{|z|}{|\Im z| } \right) = N^{-1+\epsilon}\frac{|z||w|}{|\Im z||\Im w| }\min(\frac{t_0}{t}, \frac{s_0}{s}).$
    
    \begin{multline}
    \int_0^{\infty}\int_0^{\infty} \left|\frac{\partial}{\partial z} \frac{\partial}{\partial w}\mathcal{L}_{N, k}(z, t, w, s) -\frac{\partial}{\partial z} \frac{\partial}{\partial w}\mathcal{L}^0_{N, k}(z, t, w, s)\right|\diff t \diff s\\ \leq O(1) N^{\alpha/2-2+\epsilon} \frac{|z||w|}{|\Im z|^2 |\Im w|^2}\left( \int_{t_0}^{+\infty} \int_0^{t\frac{s_0}{t_0}}\frac{t_0}{t}\exp\left(\frac{-t|\Im z|-s|\Im w|}{2}\right) \diff s \diff t  \right. \\
    \left. + \int_{s_0}^{+\infty} \int_0^{s\frac{t_0}{s_0}}\frac{s_0}{s}\exp\left(\frac{-t|\Im z|-s|\Im w|}{2}\right) \diff t \diff s \right) \\ 
    \leq O(1) N^{\alpha/2-2+\epsilon} \frac{|z||w|}{|\Im z|^2 |\Im w|^2}
    \left(\frac{1}{|\Im w|}\int_{t_0}^{+\infty}\frac{t_0}{t}\exp\left(\frac{-t |\Im z|}{2}\right) \diff t\right.\\ \left.+\frac{1}{|\Im z|}\int_{s_0}^{+\infty}\frac{s_0}{s}\exp\left(\frac{-s |\Im w|}{2}\right) \diff s \right)\\
    \leq O(1)N^{\alpha/2 - 2+\epsilon} \frac{|z||w|}{|\Im z|^3 |\Im w|^3}.
    \end{multline}  
\end{proof}

Suppose, that there exist a function $\mathcal{F}(z, t, w, s)>0$ such that\\
$\left|\frac{\partial}{\partial z} \frac{\partial}{\partial w}\mathcal{L}^0_{N, k}(z, t, w, s)\right|<\mathcal{F}(z, t, w, s)$ everywhere with probability 1 for all $N, k$ and 
\begin{equation}
    \int_0^{\infty}\int_0^{\infty} \mathcal{F}(z, t, w, s) \diff t \diff s <D(z, w).
    \label{eqn:bound_integrable_function}
\end{equation}
Then condition \eqref{eqn:bound} would automatically hold and convergence in \eqref{eqn:convergence} will follow from 
\begin{multline}
    \int_{s_1}^{s_2}\int_{t_1}^{t_2}\frac{\partial}{\partial z} \frac{\partial}{\partial w}\mathcal{L}_{N, k}(z, t, w, s)  \diff t \diff s  = \frac{\partial}{\partial z}\frac{\partial}{\partial w} \int_{s_1}^{s_2}\int_{t_1}^{t_2} \mathcal{L}_{N, k}(z, t, w, s)  \diff t \diff s \\ \stackrel{\Prob}{\rightarrow}\frac{\partial}{\partial z} \frac{\partial}{\partial w} \int_{s_1}^{s_2}\int_{t_1}^{t_2}\mathcal{L}(z, t, w, s)  \diff t \diff s = 
    \int_{s_1}^{s_2}\int_{t_1}^{t_2}\frac{\partial}{\partial z} \frac{\partial}{\partial w}\mathcal{L}(z, t, w, s)  \diff t \diff s \label{eqn:convergence_cut} 
\end{multline}
uniformly in $k$ for all fixed $0<t_1<t_2$ and $0<s_1<s_2.$

\begin{lemma}
\label{lemma:uniform_bounded}
For all $N, k$ for $\mathcal{L}_0(z, t, w, s)$ defined as above
\begin{equation*}
    \left|\frac{\partial}{\partial z}\frac{\partial}{\partial w}\mathcal{L}^0_{N, k}(z, t, s, w)\right| \leq \mathcal{F}(z, t, w, s),
\end{equation*}
where
\begin{multline*}
    \mathcal{F}(z, t, w, s):  =  \\ O(1)\exp\left(\frac{-t|\Im z| - s |\Im w|}{2}\right)\left(t^{\alpha/4-1}|z|^{\alpha/4}| \Im z|^{\alpha/4-1}\times s^{\alpha/4-1}|w|^{\alpha/4}| \Im w|^{\alpha/4-1}\right).
\end{multline*}
\end{lemma}
\begin{proof}

\begin{multline}
\Cov_{\mathbf{x}_k}\left[\frac{\exp\left(\sgn_{\Im z}\iunit t f_{N, k}(z)\right)}{t}, \frac{\exp\left(\sgn_{\Im w} \iunit s \underline{f_{N, k}}(w)\right)}{s}\right]=\\ \frac{\prod_1^P\phi_N(t u_i+s v_i) -\prod_1^P\phi_N(t u_i)\phi_N (s v_i)}{ts},
\label{eqn:expansion}
\end{multline}

where $u_i: = \sgn_{\Im z}  z \G_{N,k}(z)_{ii}$ and $v_i:= \sgn_{\Im w} w\underline{\G_{N,k}}(w)_{ii}.$ 
Denote 
\begin{equation}
    \ell^{(i)}_{N, k}(z, t, w, s):=N^{\alpha/2}\frac{\phi_N(t u_i+s v_i)-\phi_N(t u_i)\phi_N(s v_i)}{ts}
\end{equation}
and 
\begin{equation}
     r^{(i)}_{N, k}(z, t, w, s):=\prod_{j=1}^{i-1}\phi_N(t u_j+s v_j)\prod_{j=i+1}^{P}\phi_N(t u_j)\phi_N(s v_j).
\end{equation}
We can rewrite
\begin{equation}
    \mathcal{L}_{N, k}(z, t, w, s)=\frac{\exp\left(\iunit \sgn_{\Im z}tz+\iunit \sgn_{\Im w}sw\right)\sum_{i=1}^{P} \left(\ell^{(i)}_{N, k}(z, t, w, s)\times r^{(i)}_{N, k}(z, t, w, s)\right)}{N}.
\end{equation}
By Lemma \ref{lemma:imsgn} $\Im u_i < 0$ and $\Im v_i<0,$ therefore 
\begin{equation}
    \left|r^{(i)}_{N, k}(z, t, w, s)\right|\leq 1 .
\end{equation}
Next, we will prove that
\begin{equation}
    \left|\ell^{(i)}_{N, k}(z, t, w, s)\right|\leq O(1)~t^{\alpha/4-1}|z|^{\alpha/4}|\Im z|^{-\alpha/4}\times s^{\alpha/4-1}|w|^{\alpha/4}|\Im w|^{-\alpha/4}.
\end{equation}
By Cauchy-Schwartz inequality for complex random variables $X$ and $Y$ 
\begin{multline}
|\Cov_{\x_k}(X, Y)| = |\Exp_{\x_k}(X - \Exp_{\x_k} X, Y - \Exp_{x_k}Y)|\leq \\ \sqrt{\left(\Exp_{\x_k}X\overline{X} - \Exp_{\x_k} X\overline{\Exp_{\x_k}X}\right)\left(\Exp_{\x_k}Y\overline{Y} - \Exp_{\x_k} Y\overline{\Exp_{\x_k}Y}\right)}.
\end{multline}
Thus 
\begin{multline}
\left|\phi_N(tu_i+sv_i)-\phi_N(tu_i)\phi_N(sv_i)\right|\leq\\ \sqrt{\left(\phi_N(2t\iunit \Im u_i) - \phi_N(t u_i)\overline{\phi_N(t u_i)} \right)\left(\phi_N(2s\iunit \Im v_i) - \phi_N(s v_i)\overline{\phi_N(s v_i)} \right)}.
\end{multline}

The 5-th part of Lemma \ref{lem:truncbound} allows the following estimate for $t\leq t_0$:

$$\phi_N(2t\iunit \Im u_i) - \phi_N(t u_i)\overline{\phi_N(t u_i)} = O\left(\frac{|tz|^{\alpha/2}}{|\Im z|^{\alpha/2} N^{\alpha/2}}\right),$$
which leads to 
\begin{equation}
\left|\phi_N(tu_i+sv_i)-\phi_N(tu_i)\phi_N(sv_i)\right|\leq O\left(\frac{|tz|^{\alpha/4}|sw|^{\alpha/4}}{|\Im z|^{\alpha/4}|\Im w|^{\alpha/4}  N^{\alpha/2}}\right).
\label{eqn:smallestimate}   
\end{equation}

Thus, using the bounds on the support of $\mathcal{L}_{N, k}^0(z, t, w, s)$ we can conclude, that
\begin{multline}
    \left|\mathcal{L}^0_{N, k}(z, t, s, w)\right| \leq \\  O(1) \exp\left(-t|\Im z| - s |\Im w|\right)(t^{\alpha/4-1}|\Im z|^{-\alpha/4-1}\times s^{\alpha/4-1}|\Im w|^{-\alpha/4-1}).
\end{multline}

Applying Cauchy inequality (Lemma~\ref{lemma:cauchy_inequality}), we bound
\begin{multline}
    \left|\frac{\partial}{\partial z} \frac{\partial}{\partial w}\mathcal{L}^0_{N, k}(z, t, s, w)\right| \leq \\  O(1) \exp\left(\frac{-t|\Im z| - s |\Im w|}{2}\right)(t^{\alpha/4-1}|z|^{\alpha/4}|\Im z|^{-\alpha/4-1}\times s^{\alpha/4-1}|w|^{\alpha/4}|\Im w|^{-\alpha/4-1}).
\end{multline}
\end{proof}
Notice, that
\begin{multline}
    \int_0^{\infty} \int_0^{\infty} \exp\left(\frac{-t|\Im z| - s |\Im w|}{2}\right)(t^{\alpha/4-1}|\Im z|^{\alpha/4-1}\times s^{\alpha/4-1}|\Im w|^{\alpha/4-1}) \diff t \diff s \leq  \\ 8 \frac{1}{|\Im z||\Im w|}\Gamma(\alpha/4)^2, 
\end{multline}

which means, that the input of $\mathcal{L}^0_{N, k}$ into $D(z, w)$ does not exceed $O(1)\frac{|z|^{\alpha/4}|w|^{\alpha/4}}{|\Im z|^{\alpha/2+1}|\Im w|^{\alpha/2+1}},$ which finishes the proof of the bound \eqref{eqn:bound}.

\subsubsection{Proof of convergence \eqref{eqn:convergence_cut}}

Recalling the proof of Lemma \ref{lemma:uniform_bounded}, the limit \eqref{eqn:convergence_cut} will follow from the Lemma below. 
\begin{lemma}
Fix $0<t_1<t_2$ and $0<s_1<s_2$ and $z, w \in \C\backslash\R$
\begin{equation}
     \max_{\substack{t\in[t_1, t_2]\\s\in[s_1, s_2]}} \left|\mathcal{L}_{N, k}(z, t, w, s)-\mathcal{L}(z, t, w, s)\right|  \underset{N \rightarrow \infty}{\rightarrow} 0.    
\end{equation}
uniformly on $k$ in probability.
\label{lemma:convergence}
\begin{proof}
Using expansion \eqref{eqn:expansion}, by Lemma \ref{lemma:convergenceofaverages} it is enough to prove, that 
\begin{equation}
    \max_{\substack{t\in[t_1, t_2]\\s\in[s_1, s_2]}}\left[\ell ^{(i)}_{N, k}(z, t, w, s)-\ell(z, t, w, s)\right] \underset{N \rightarrow \infty}{\rightarrow} 0,
    \label{eqn:ellconv}
\end{equation}
and 
\begin{equation}
    \max_{\substack{t\in[t_1, t_2]\\s\in[s_1, s_2]}}\left[r^{(i)}_{N, k}(z, t, w, s) -r(z, t, w, s)\right]\underset{N \rightarrow \infty}{\rightarrow} 0. 
    \label{eqn:rconv}
\end{equation}
in probability uniformly over $i, k.$
Firstly, we will prove the limit \eqref{eqn:ellconv}.
Using the part 5 of Lemma \ref{lem:truncbound} we get, the asymptotic decomposition below:
\begin{multline}
 \frac{ts}{N^{\alpha/2}} \times \ell ^{(i)}_{N, k}(z, t, w, s) =\phi_N(tu_i+sv_i)-\phi_N(tu_i)\phi_N(sv_i)\\=
1-\frac{(tu_i+sv_i)}{N}+c\frac{\left(tu_i+sv_i\right)^{\alpha/2}}{N^{\alpha/2}}+\frac{\left|tu_i+sv_i\right|^{\alpha/2}}{N^{\alpha/2}}\mathfrak{E}_N\left(\frac{tu_i+sv_i}{N}\right)\\-
\left(1-\frac{tu_i}{N}+c\frac{\left(tu_i\right)^{\alpha/2}}{N^{\alpha/2}}+\frac{\left|tu_i\right|^{\alpha/2}}{N^{\alpha/2}}\mathfrak{E}_N\left(\frac{tu_i}{N}\right)\right)\\ \times \left(1-\frac{sv_i}{N}+c\frac{\left(sv_i\right)^{\alpha/2}}{N^{\alpha/2}}+\frac{\left|sv_i\right|^{\alpha/2}}{N^{\alpha/2}}\mathfrak{E}_N\left(\frac{sv_i}{N}\right)\right)\\=
\frac{1}{N^{\alpha/2}}c\left(\left(tu_i+sv_i\right)^{\alpha/2}-tu_i^{\alpha/2}-sv_i^{\alpha/2}\right)+o(1)\frac{1}{N^{\alpha/2}},
\end{multline}
and $o(1)\leq c_N(z, t_1, t_2, w, s_1, s_2) \underset{N \to \infty}{\rightarrow} 0.$
\begin{remark}
    For the fixed $z, t_1, t_2, w, s_1, s_2$ the order of convergence of $c_N$ can be estimated from the top through the $\max$ of $\mathfrak{E}_N(\cdot)$ on the half-ball with centre in 0 and radius $\frac{t_2}{N\left|\Im z\right|}+\frac{s_2}{N\left|\Im w\right|}.$
\end{remark}
We can make further expansion:
\begin{multline*}\left(tu_i+sv_i\right)^{\alpha/2}-\left(tu_i\right)^{\alpha/2}-\left(sv_i\right)^{\alpha/2}=\\
\left(K(t, z)+K(s, w)\right)^{\alpha/2}-K(t, z)^{\alpha/2}-K(s, w)^{\alpha/2} + \Delta_{i, k}(z, t, w, s),
\end{multline*}
where 
\[
\Bigl|\Delta_{i, k}(z, t, w, s)\Bigr|\leq 4 (t_2+s_2)^{\alpha/2}\biggl(\Bigl|z(\G_{N, k})_{ii}(z)-zm_y(z)\Bigr|+\Bigl|w(\underline{\G_{N, k}})_{ii}(w)-zm_y(w)\Bigr|\biggr).
\] The random variable $\left(\G_{N, k}\right)_{ii}(z)$ uniformly on $(i, k)$ converges in probability to $m(z).$

Next, we prove convergence \eqref{eqn:rconv}. \begin{multline}
r_{N, k}^{(i)}(z, t, w, s)= \prod_{j=1}^{i-1}\phi_N(u_j+v_j)\prod_{j=i+1}^{P}\phi_N(u_j)\phi_N(v_j)=\\
\exp\left(\sum_{j=1}^{i-1}\ln \phi_N(u_j+v_j)+\sum_{j=i+1}^{P}\left[\ln \phi_N(u_j)+\ln \phi_N(v_j)\right]\right).
\end{multline}
When $|z|<\frac{1}{10}.$ $|\ln(1+z)-z|<10 |z|^3$ 
For bounded $u$
\[\ln \phi_N(u)=-\frac{\iunit u}{N}+O\left(\frac{1}{N^{\alpha/2}}\right),\]
which leads to 
\[
\prod_{j=1}^{i-1}\phi_N(u_j+v_j)\prod_{j=i+1}^{P}\phi_N(u_j)\phi_N(v_j)=\exp\left(-\iunit \frac{\sum_{j=1, j \neq i}^P(u_j+v_j)}{N}\right)(1+o(1)).
\]
Recalling the definitions of $u_i$ and $v_i$ we get that
\begin{multline*}
-\iunit \frac{\sum_{j=1, j \neq i}^P(u_j+v_j)}{N}=-\iunit t\sgn_{\Im z}\frac{1}{N}z\Tr\G_{N, k}(z)-\iunit s\sgn_{\Im w}\frac{1}{N}w\Tr\underline{\G_{N, k}}(w) \\
= -yK(z, t)-y K(s, w) -\iunit t z \sgn_{\Im z} M_{N}(z)- \iunit s w \sgn_{\Im w}\underline{M_{N}}(w)\\ +O\left(\frac{t|z|}{N|\Im z|}\right) + O\left(\frac{s|w|}{N|\Im w|}\right),
\end{multline*}
where $M(z):= \frac{1}{N}\Tr \G_N(z) - y m_y(z).$
Marchenko-Pastur law yields, that $M_{N}(z)\to 0$ in probability, as well as $\underline{M_{N}}(w).$ Thus, \eqref{eqn:rconv} will hold, and convergence, will obviously be uniform. 
\end{proof}

\end{lemma}

\section{Proof of the Theorem \ref{thm:covkernrewrite}}
\label{sec:covcalculation}
To calculate the integral from Theorem \ref{thm:main} we will follow \cite{LM22}.

Using Lemma \ref{lemma:raising_into_power} it is possible to rewrite the following equation as an integral:

\begin{multline}
 \left(K(z, t)+K\left(w, s\right)\right)^{\alpha / 2}-K(z, t)^{\alpha/2}-K(w, s)^{\alpha/2}\\
 =  \frac{1}{\Gamma(-\alpha/2)} \int_0^{\infty} \frac{ \left(\exp\left(-rK(z, t)\right)-1\right)\left(\exp\left(-r K(w, s)\right) - 1\right)}{r^{\frac{\alpha}{2}+1}} \mathrm{d} r.
 \label{eqn:rrewriting}
 \end{multline}

 Thus, denoting 
 \begin{equation}
     \mathfrak{k}(z, t, r):=\frac{\left(\exp\left(-rK(z, t)\right)-1\right)\exp(\sgn_z\iunit t z - y K(z, t))}{t},
 \end{equation}
  it is possible to rewrite 

  \begin{equation}
      \mathcal{L}(z, t, w, s) = yc \frac{1}{\Gamma(-\alpha/2)} \int _0^\infty \frac{1}{r^{\frac{\alpha}{1}+1}}  \mathfrak{k}(z, t, r)  \mathfrak{k}(w, s, r) \diff r.
  \end{equation}

 and 
\begin{multline}
      C(z, w) = \int_0^\infty \int_0^\infty \frac{\partial}{\partial z}\frac{\partial}{\partial w}\mathcal{L}(z, t, w, s) \diff s \diff t= \\ yc \frac{1}{\Gamma(-\alpha/2)} \int _0 ^\infty \int _0 ^\infty \int _0^\infty \frac{1}{r^{\frac{\alpha}{1}+1}}  \frac{\partial}{\partial z}\mathfrak{k}(z, t, r)  \frac{\partial}{\partial w} \mathfrak{k}(w, s, r) \diff r \diff s \diff t = \\ \\ yc \frac{1}{\Gamma(-\alpha/2)} \int _0 ^\infty  \frac{1}{r^{\frac{\alpha}{2}+1}}  \frac{\partial}{\partial z}\left(\int_0^\infty \mathfrak{k}(z, t, r) \diff t \right) \frac{\partial}{\partial w} \left( \int_0^\infty \mathfrak{k}(w, s, r) \diff s\right) \diff r .
      \label{eqn:rewritingwithr}
  \end{multline}
 
% From \eqref{eqn:mpstieltj} the following equation for $m_y(z)$
% \begin{equation}
%     z - yzm_y(z) = \frac{z}{(1-y)+ z y m_y(z)}. 
% \end{equation}
% This way, 
% \begin{multline}
% \mathcal{L}(t, s, z, w) =  y c \frac{1}{\Gamma[-\alpha/2]} \int_0 ^ \infty \frac{1}{r^{\frac{\alpha}{2}+1}} \times \\
% \frac{1}{t}\left[\exp\left(\operatorname{sgnIm}(z)it\left(- rzm_y(z)+ \frac{z}{(1-y)+ z y m_y(z)}\right)\right) - \exp \left( \frac{z \operatorname{sgnIm}(z)it }{(1-y)+ z y m_y(z)}\right)\right]\times \\
% \frac{1}{s}\left[\exp\left(\operatorname{sgnIm}(w)is\left( - rwm_y(w)+ \frac{w}{(1-y)+ w y m_y(w)}\right)\right) - \exp \left( \frac{w \operatorname{sgnIm}(w)is }{(1-y)+ w y m_y(w)}\right)\right] \mathrm{d} r 
% \end{multline}
% This formula allows us to separate integral by $s$ and integral by $t$. 

 Lemma \ref{lemma:subtract} allows us to conclude, that 
\begin{equation}
    \int_0^\infty \mathfrak{k}(z, t, r) \diff t =\\
    -\log \left( 1-\frac{rm_y(z)}{1-ym_y(z)} \right).
    \label{eqn:compkfrak}
\end{equation}
Recalling equation \eqref{eqn:mpstieltj},
\begin{equation}
    \frac{1}{1 - ym_y(z)} = (1-y)+ z y m_y(z). 
\end{equation}
Substituting it into \eqref{eqn:compkfrak} we get, that
\begin{equation}
    \int_0^\infty \mathfrak{k}(z, t, r) \diff t =\\
    -\log \left( 1-rm_y(z)\left((1-y)+ z y m_y(z)\right) \right).
\end{equation}
Similarly, by \eqref{eqn:mpstieltj} holds  $m_y(z)\left((1-y)+zym_y(z)\right) = zm_y(z) -1 ,$  so we can simplify
\begin{equation}
    \int_0^\infty \mathfrak{k}(z, t, r) \diff t =\\
    -\log \left( 1-r(zm_y(z) -1) \right),
\end{equation}
from where we can deduce
\begin{equation}
    \frac{\partial}{\partial z}\int_0^\infty \mathfrak{k}(z, t, r) \diff t =
    r\frac{ \frac{\partial}{\partial z}zm_y(z)}{1-r\left(zm_y(z) -1\right)}.
\end{equation}

Substituting this into \eqref{eqn:rewritingwithr} we get that

\begin{equation}
    C(z, w) = \frac{yc}{\Gamma(-\alpha/2)}\frac{\partial}{\partial z}zm_y(z) \frac{\partial}{\partial w}wm_y(w)\int _0 ^\infty  \frac{r^{1-\alpha/2}}{(1-r\left(zm_y(z) -1\right))(1-r\left(wm_y(w) -1\right))} \diff r.
    \label{eqn:almostthere}
\end{equation}
% \begin{multline}
%     \int_{t, s>0}  \mathcal{L}(t, s, z, w)\mathrm{d} t \mathrm{~d} s =  y c \frac{1}{\Gamma[-\alpha/2]} \int_0^{\infty} \frac{1}{r^{\frac{\alpha}{2}+1}} \times\\
%     \left[ \log\left(-rz m_y(z) +\frac{z}{(1-y)+ z y m_y(z)}\right) - \log \left(\frac{z}{(1-y)+ z y m_y(z)}\right) \right] \times \\
%      \left[ \log\left(-rw m_y(w) +\frac{w}{(1-y)+ w y m_y(w)}\right) - \log \left(\frac{w}{(1-y)+ w y m_y(w)}\right) \right] \mathrm{d} r.
% \end{multline}
%  To get the covariance kernel we need to differentiate this expression by $z$ and $w.$ As soon as variables $z$ and $w$ are now separated, it will follow from
%  \begin{equation}
%  \begin{aligned}
%  \frac{\partial}{\partial z} \left[\log\left(-rz m_y(z) +\frac{z}{(1-y)+ z y m_y(z)}\right) - \log\left(\frac{z}{(1-y)+ z y m_y(z)}\right) \right]= \\
%  - r \cdot \frac{1}{z} \cdot \frac{\frac{\partial}{\partial z}\left( m_y(z)\left( (1-y) + z y m_y(z)\right)\right)}{1- r \left(m_y(z)\left( (1-y) + z y m_y(z) \right)\right)}.
% \end{aligned}
%  \end{equation}

 The Lemma \ref{lemma:int_r} allows us to calculate 
 \begin{multline}
     \int _0 ^\infty  \frac{r^{1-\alpha/2}}{(1-r\left(zm_y(z) -1\right))(1-r\left(wm_y(w) -1\right))} \diff r  \\ = \frac{\pi}{\sin\left(\pi\frac{\alpha}{2}\right)}\frac{\left(-1+zm_y(z)\right)^{\alpha/2-1}- \left(-1+wm_y(w)\right)^{\alpha/2-1}}{zm_y(z) - w m_y(w)}.
 \end{multline}

 Combining it with \eqref{eqn:almostthere} and using $\frac{\pi}{\sin\left(\pi\frac{\alpha}{2}\right)} = \Gamma\left(-\frac{\alpha}{2}\right)\Gamma\left(1+\frac{\alpha}{2}\right),$ we conclude that
 
 \begin{multline}
    C(z, w) = -yc \Gamma\left(1+\frac{\alpha}{2}\right)\frac{\partial}{\partial z}\left(zm_y(z)\right)\frac{\partial}{\partial w}\left(wm_y(w)\right) \\ 
 \times \frac{\left(-1+zm_y(z)\right)^{\alpha/2-1}- \left(-1+wm_y(w)\right)^{\alpha/2-1}}{zm_y(z) - w m_y(w)}.
\end{multline}

\section{Proof of \ref{thm:theoremforoverlapping}}
\label{sec:overlapping}
In this section, we will show how to adapt the proof of Theorem \ref{thm:main} so that it works for overlapping half-heavy-tailed random matrices, where the number of overlapping rows and columns is proportional to $N.$ The proof of the truncation and diagonalization stages does not differ from one in previous sections. Nevertheless, to move to the computation of the limit we need to be careful with the number of non-zero terms in the martingale sum and with the number of resolvent diagonal elements involved.

Denote 
\begin{equation}
    Y_{k}^{[i]}(z) :=\frac{1}{ N^{1-\alpha/4}}\left(\mathbb{E}_{k}-\mathbb{E}_{k-1}\right)\Tr\G_{\mathbf{A}^{[i]}_N}(z).
\end{equation} 
Similarly to Section \ref{sec:mart_decomposition} it is enough to prove that 

\begin{equation}
    \sum_{k\in \mathcal{Q}_i \cap \mathcal{Q}_j}\mathbb{E}_{k-1}\left[Y_k^{[i]}(z) Y_k^{[j]}(w)\right]\to C_{i, j}(z, w).
\end{equation}
Denote $\left(\mathbf{x}_k\mid_{\mathcal{P}_i}\right)$ the projection of the vector $\mathbf{x}_k$ on coordinates out of the set $\mathcal{P}_i.$ Then,
``diagonalization" term $\tilde{Y}_k^i(z)$ can be defined the following way:
\begin{equation}
    \tilde{Y}_k^{[i]}(z):=\begin{cases} 0, &\textit{$k\notin\mathcal{Q}_i$}\\
    \frac{\partial}{\partial z} \frac{1}{N^{1-\alpha/4}} \left(\mathbb{E}_{k}-\mathbb{E}_{k-1}\right) \log\left|\Tilde{g}^{[i]}_{N, k}(z)\right|^2, &\textit{otherwise},
                        \end{cases}
\end{equation}
where $\Tilde{g}^{[i]}_{N, k}(z):=z- \frac{1}{N}\left(\mathbf{x}_k\mid_{\mathcal{P}_i}\right)^* z \operatorname{diag}\left[\mathbf{G}^{[i]}_{N, k}(z)\right] \left(\mathbf{x}_k\mid_{\mathcal{P}_i}\right).$
Further, for $k\in \mathcal{Q}_i \cap \mathcal{Q}_j$ we will denote
\begin{multline}
    \mathcal{L}^{[i, j]}_{N, k}(z, t, w, s)\\ := \frac{\left|\mathcal{Q}_i \cap \mathcal{Q}_j\right|}{N}\cdot N^{\alpha/2-1}\frac{\Cov_{\mathbf{x}_k}\left[\exp\left(\sgn \Im z \iunit t \tilde{g}^{[i]}_{N, k}(z)\right), \exp\left(\sgn \Im w \iunit s \underline{\tilde{g}^{[j]}_{N, k}}(w)\right)\right]}{ts}\\=
    \frac{\left|\mathcal{Q}_i \cap \mathcal{Q}_j\right|}{N^2}\exp\left(\iunit \sgn_{\Im z}tz+\iunit \sgn_{\Im w}sw\right)\\ \times \sum_{m\in \mathcal{P}_i\cap\mathcal{P}_j} \left[\ell^{(m), [i, j]}_{N, k}(z, t, w, s)\times r^{(m)[i, j]}_{N, k}(z, t, w, s)\right],
\end{multline}

where
\begin{gather}
    u^{[i]}_m:= - \sgn_{\Im z} \iunit z \left(\mathbf{G}^{[i]}_{N, k}(z)\right)_{m, m}\\
    v^{[j]}_m:= - \sgn_{\Im z} \iunit w \left(\underline{\mathbf{G}^{[j]}_{N, k}}(w)\right)_{m, m}\\
    \ell^{(m) [i, j]}_{N, k}(z, t, w, s):=N^{\alpha/2}\frac{\phi_N(t u^{[i]}_m+s v^{[j]}_m)-\phi_N(t u^{[i]}_m)\phi_N(s v^{[j]}_m) }{ts} ,
\end{gather}
and
\begin{multline}
    r^{(m)[i, j]}_{N, k}(z, t, w, s) := \prod_{\substack{n<m \\n\in \mathcal{P}_i\cap\mathcal{P}_j} }\phi_N(t u^{[i]}_n+s v^{[j]}_n)\prod_{\substack{n>m\\ n\in \mathcal{P}_i\cap\mathcal{P}_j}}\phi_N(t u^{[i]}_n)\phi_N(s v^{[j]}_n)\\ \times \prod_{n\in \mathcal{P}_i\backslash \mathcal{P}_j}\phi_N(t u^{[i]}_n) \prod_{n\in \mathcal{P}_j\backslash \mathcal{P}_i}\phi_N(s v^{[j]}_n) .
\end{multline}
As in Subection \ref{sec:computation_of_limit}, we see that
\begin{equation}
    \ell^{(m) [i, j]}_{N, k}(z, t, w, s) \to c \frac{\left(K^{[i]}(z, t)+K^{[j]}(w, s)\right)^{\alpha/2}-K^{[i]}(z, t)^{\alpha/2}-K^{[j]}(w, s)^{\alpha/2}}{ts}
\end{equation}
and 
\begin{equation}
    r^{(m)[i, j]}_{N, k}(z, t, w, s) \to \exp\left(-p_iK^{[i]}(z, t) -p_jK^{[j]}(w, s)\right),
\end{equation}
where $K^{[i]}(z, t)=t\sgn_z\iunit z s^{[i]}(z), $ and $s^{[i]}(z)$ denotes the limit of the diagonal elements of the resolvent of $\BA_N^{[i]}.$
Notice, that
$$
z - \BA^{[i]}_N = z  - \frac{\X_N^{[i]}\X_N^{[i]*}}{N} = z  - \frac{\left|\mathcal{Q}_i\right|}{N}\cdot\frac{\X_N^{[i]}\X_N^{[i]*}}{\left|\mathcal{Q}_i\right|}=\frac{\left|\mathcal{Q}_i\right|}{N}\left(z\frac{N}{\left|\mathcal{Q}_i\right|}-\frac{\X_N^{[i]}\X_N^{[i]*}}{\left|\mathcal{Q}_i\right|}\right).
$$
Thus, 
$$s^{[i]}(z) = \frac{1}{q_i}m_{\frac{p_i}{q_i}}\left(\frac{z}{q_i}\right).$$

% \begin{multline}
%      C(z, w) =  yc\cdot \frac{1}{zw} \cdot \frac{1}{\Gamma\left[-\alpha/2\right]}\cdot\frac{\pi}{\sin\left(\frac{\pi\alpha}{2}\right)}\cdot \frac{\partial}{\partial z}\left( m_y(z)\left( (1-y) + z y m_y(z)\right)\right) \times\\
%      \frac{\partial}{\partial w}\left( m_y(w)\left( (1-y) + w y m_y(w)\right)\right)\times \\
%    \frac{\left( m_y(z)\left( (1-y) + z y m_y(z)\right)\right)^{\alpha/2-1} - \left( m_y(w)\left( (1-y) + w y m_y(w)\right)\right)^{\alpha/2-1}}{ m_y(z)\left( (1-y) + z y m_y(z)\right) -  m_y(w)\left( (1-y) + w y m_y(w)\right)}
% \end{multline}

% One can easily check, that 

% $m_y(z)\left((1-y)+zym_y(z)\right) = zm_y(z) -1 .$ Also, by the property of Gamma function $\frac{\pi}{\sin\left(\pi\frac{\alpha}{2}\right)} = \Gamma\left(-\frac{\alpha}{2}\right)\Gamma\left(1+\frac{\alpha}{2}\right).$ Thus, we can simplify the kernel the following way:

% \begin{equation}
%     C(z, w) = -yc \frac{1}{zw} \Gamma\left(1+\frac{\alpha}{2}\right)\frac{\partial}{\partial z}\left(zm_y(z)\right)\frac{\partial}{\partial w}\left(wm_y(w)\right)\frac{\left(-1+zm_y(z)\right)^{\alpha/2-1}- \left(-1+wm_y(w)\right)^{\alpha/2-1}}{zm_y(z) - w m_y(w)}.
% \end{equation}

 \appendix
\section{Appendix}
We use the following rank inequalities adapted from \cite[Theorem~A.44]{BS10}, where proofs may be found.

\begin{lemma}\cite[Theorem~A.44]{BS10}
\label{lem:wishartdiff}
Let $\X_N$ and $\hat{\X}_N$ be two $P\times N$ complex matrices and let 
$F(\cdot)$ and $\hat{F}(\cdot)$ be the cumulative distribution functions of the
the empirical spectral measures of $\X_N\X_N^*$ and $\hat{\X}_N\hat{\X}_N^*$ respectively:
\begin{equation*}
\begin{aligned}
 F(x) := \frac{\#\big\{j : \lambda_j\left(\frac{\X_N\X_N^*}{N}\right) \leq x\big\}}{P} ,\\ 
\hat{F}(x) := \frac{\#\big\{j : \lambda_j\left(\frac{\hat{\X}_N\hat{\X}_N^*}{N} \right) \leq x\big\}}{P} 
\end{aligned}
\qquad x \in \R .
\end{equation*}

Then the following inequality holds
\[
\sup_{x\in \R} \left|F(x) - \hat{F}(x)\right| \leq 
\frac{1}{P}\rank(\X_N-\hat{\X}_N).
\]
\end{lemma}

The following Corollary of the above Lemma is used to truncate and
centre the original matrix $\X_N$.
\begin{corollary}
\label{cor:resolvrank}
Under the same assumptions as Lemma~\ref{lem:wishartdiff}, we have 
the following bound on the resolvent for all $z \in \C\backslash\R$,
\[
\left| \Tr\left[ \left(z  - \frac{\X \X^*}{N}\right)^{-1}\right]
- \Tr\left[ \left(z - \frac{\hat{\X} \hat{\X}^*}{N}\right)^{-1}\right]\right| \leq 
\frac{\pi}{|\Im z|}\rank(\X - \hat{\X}) .
\]
\end{corollary}
\begin{proof}
Using equation \eqref{eqn:reslss} and integration by parts, we can get 
\begin{multline*}
\frac{1}{P}\left| \Tr\left[ \left(z  - \frac{\X \X^*}{N}\right)^{-1}\right]
- \Tr\left[ \left(z - \frac{\hat{\X} \hat{\X}^*}{N}\right)^{-1}\right]\right|  \\
= \left|\int_{-\infty}^{+\infty}\frac{1}{z-\lambda} \diff \left(F(\lambda)-\hat{F}(\lambda)\right)\right| = \left|\int_{-\infty}^{+\infty}\frac{F(\lambda)-\hat{F}(\lambda)}{(z-\lambda)^2}\diff \lambda\right| \\ \leq \int_{-\infty}^{+\infty}\frac{1}{\left|z-\lambda\right|^2}\diff \lambda \times \sup_{x\in \R} \left|F(x) - \hat{F}(x)\right|.
\end{multline*}
Note that for any $z\in \C\backslash\R$
\begin{equation*}
\int_{-\infty}^{+\infty}\frac{1}{\left|z-\lambda\right|^2}\diff \lambda =\int_{-\infty}^{+\infty}\frac{1}{\left(\Re z-\lambda\right)^2 +\Im z ^2} \diff \lambda = \int_{-\infty}^{+\infty}\frac{1}{x^2 +\Im z ^2}  \diff x = \frac{\pi}{\Im z}.
\end{equation*}
Combing equations above with Lemma~\ref{lem:wishartdiff} we get the statement of the Corollary. 
\end{proof}

\begin{lemma} Let $\X$ be a $T\times S$ matrix, and $\x_i$ is its $i$-th column. Then 
\begin{equation}
    \left(\frac{1}{z - \X^*\X}\right)_{ii}=\frac{1}{z-\x_i^*\frac{z}{z-\left(\X\X^*-\x_i\x^*_i\right)}\x_i}
\end{equation}
\label{lemma:diagonalsttransform}
\end{lemma}
\begin{lemma}[Marchenko-Pastur law application]
    For the random matrix $\X_N$ as in Definition \ref{def:hlfhvywish} , 
    \begin{equation}  \left(\frac{1}{z - \frac{\X_N\X_N^*}{N}}\right)_{ii} \underset{D \rightarrow \infty}{\rightarrow} \frac{1}{z - \frac{z}{y}m_{1/y}\left(\frac{z}{y}\right)}. \end{equation}
    uniformly on $i$ in probability.
\end{lemma}
\begin{remark}
Notice, that equation \eqref{eqn:sttransformquadr} leads to 
$$ \frac{1}{z - \frac{z}{y}m_{1/y}\left(\frac{z}{y}\right)}
=\frac{1}{z-\left((1-y)-zym_y(z)\right)}=m_y(z).$$
Thus, 
\begin{equation}  \left(\frac{1}{z - \frac{\X_N\X_N^*}{N}}\right)_{ii} \underset{D \rightarrow \infty}{\rightarrow} m_y(z). 
\end{equation}
\end{remark}

\begin{lemma}
\label{lemma:imsgn}
For $\mathbf{A}=\mathbf{X}\mathbf{X}^*,$ where $\mathbf{X}$ is any matrix with at least $1$ non-zero element
\begin{enumerate}
    \item $\sgn \Im\left( z \Tr\G_\mathbf{A}(z) \right)= -\sgn \Im z.$
    \item $ \sgn \Im\left( z \left(\G_\mathbf{A}(z)\right)_{ii} \right)= -\sgn \Im z $
\end{enumerate}
\begin{proof}
    The first part can be seen from the positivity of eigenvalues of $\mathbf{A}.$ The second part follows from  the eigenvalue positivity of the Sample Covariance matrix and Lemma \ref{lemma:diagonalsttransform}.
\end{proof}
\end{lemma}

\begin{lemma}
For any Hermitian matrix $\mathbf{A}$ holds
$$ 
\left|\G_\mathbf{A}(z)_{ii}\right|\leq\frac{1}{\left|\Im z\right|}.
$$
\label{lemma:diagest}
\end{lemma}

\begin{lemma}[Lemma 8.3 in \cite{erdos_dynamical_nodate}, Ward identity]
\label{lemma:wardidentity}
For any Hermitian matrix $\mathbf{A}$ of the size $N\times N$ holds
$$ 
\sum_{j=1}^{N}\left|\G_\mathbf{A}(z)_{ij}\right|^2=-\frac{1}{\Im z}\Im \left(\G_\mathbf{A}(z)_{ii}\right)
$$
\end{lemma}

\begin{lemma} 
\label{lemma:derivative_swap}
Suppose that $X(z)$ is a real continuously differentiable random process for $z\in \Omega,$ where $\Omega$ is some open domain, and there exist $C$ such that with probability $1$ for all $a, b \in \R$ such that $a+\iunit b = z \in \Omega$ 
\begin{equation}
\begin{cases}
    \left|\nabla_{a, b} X\left(a+\iunit b\right)\right| \leq C, \\
    X(z) \leq C
    \end{cases}
\end{equation}
Then for any $\sigma$-algebra $\mathcal{F}$
\begin{equation}
    \frac{\partial }{\partial z }\Exp\left(X(z)\mid\mathcal{F}\right) = \Exp\left(\frac{\partial}{\partial z}X(z)\mid \mathcal{F}\right).
    \end{equation}
\end{lemma}

\begin{lemma}
For each harmonic function $u$ in $\mathcal{B}(0, 1)$ holds
$$|\nabla u(0, 0)|\leq 2 \max_{\theta} |u(cos\theta, sin\theta)|  $$
\label{lemma:poissonkernel}
\end{lemma}
\begin{proof}
By Poisson integral formula 
$$u(re^{i\theta})=\frac{1}{2\pi}\int_{-\pi}^{\pi} P_r(t-\theta)u(e^{it})d t \text{ for } 0\leq r <1,$$
where $P_r(\theta)=\sum_{n=-\infty}^{\infty} r^{|n|} e^{in\theta}$ is a Poisson kernel. It is easy to see, that $\left|\frac{d}{d r}P_r(\theta)\arrowvert_{r=0}\right|\leq 2.$ Thus, 
$$ \left|\frac{d}{d r}\frac{1}{2\pi}\int_{-\pi}^{\pi} P_r(t-\theta)u(e^{it})d t \mid_{r=0} \right|\leq \max_t |u(e^{\iunit t})|$$

%This sum and its derivatives are absolutely converging. Let's choose and fix $\theta$ that corresponds to the gradient angle in $0$. Than 
$$|\nabla u(0, 0)|=\max_{\theta}\frac{d}{dr} u(re^{i\theta})\mid_{r=0}.$$
$|\frac{d}{d r}P_r(\theta)|\leq 2\text { when } r=0.$ 
\end{proof}

\begin{corollary}
Suppose, that holomorphic function $g(z)$ is defined on $\C\backslash\R$ and satisfies
$$\left|\Re g(z)\right| \leq C \frac{|z|}{|\Im z|}$$
for all $z\in \C\backslash\R,$
where $C$ is any constant not depending on $z.$Then, $$
\left|\frac{\partial}{\partial z}\Re g(z)\right|\leq 8 C \frac{|z|}{|\Im z|^2}
$$
for all $z\in \C\backslash\R.$
\label{col:derest}
\end{corollary}
\begin{proof}
Notice that $\Re g(z)$ is a harmonic function. 
Applying Lemma \ref{lemma:poissonkernel} to the ball with the centre in $z$ and of the radius $|\Im z|/2$ we can check, that
$$\left|\frac{\partial}{\partial z}\Re g(z)\right|\leq 
\left|\nabla_{x, y} \Re g(x+\iunit y)\right|\leq \frac{1}{|\Im z|/2}\max_{w\in \mathcal{B}\left(z, \frac{|\Im z|}{2} \right)}|\Re g(w)|\leq 8C\frac{|z|}{|\Im z|^2}
$$
\end{proof}

% \begin{lemma}
% $E_k(z)$ converges to $0$ in probability uniformly on $k$. 
% \label{lemma:Ekto0}
% \end{lemma}

% Consider real and imaginary parts of $E_k$ separately. For $1\leq j \leq N$ random variables  $(|x_{j, k}|^2-\frac{\sigma_N^2}{N})$ are i.i.d, real, with $0$ mean. Real and imaginary parts of $(G_{N, k})_{jj}$ are bounded with $\frac{1}{|\operatorname{Im}z|}$ (Lemma 
% \ref{lemma:diagest}). 
% Using part 4 of Lemma \ref{lem:trunc} characteristic function the $\Re E_k$  and $\Im E_k$ can be estimated through $\left(1+o\left(t\frac{\frac{1}{|\Im z|}}{N}\right)\right)^P,$ which converges to $1$ for any fixed $z.$
% Thus, characteristic functions of the $\Re E_k$ and $\Im E_k$ converge to the characteristic function of $0.$ This means that $E_k\to 0$ in probability. 

\begin{lemma}[CLT for martingales,Th. 35.12 in  Billingsley (1995)] Suppose for each $n$  $Y_{n 1}, Y_{n 2}, \ldots Y_{n r_n}$ is a real martingale difference sequence with respect to the increasing $\sigma$-field $\left\{\mathcal{F}_{n j}\right\}$ having second moments. If for each $\varepsilon>0$,
$$
\sum_{j=1}^{r_n} \mathbb{E}\left(Y_{n j}^2 I_{\left(\left|Y_{n j}\right| \geq \varepsilon\right)}\right) \rightarrow 0 \quad $$

$$\sum_{j=1}^{r_n} \mathbb{E}\left(Y_{n j}^2 \mid \mathcal{F}_{n, j-1}\right) \stackrel{i . p .}{\longrightarrow} \sigma^2,
$$
as $n \rightarrow \infty$, where $\sigma^2$ is a positive constant, then
$$
\sum_{j=1}^{r_n} Y_{n j} \stackrel{D}{\rightarrow} N\left(0, \sigma^2\right)
$$
\label{lemma:martingale_clt}
\end{lemma}

\begin{lemma}
\begin{enumerate}

\item Suppose that the array of random variables $X^{(1)}_N, X^{(2)}_N, \dots X^{(N)}_N$ is such that
\[
X^{(k)}_N\underset{N \rightarrow \infty}{\rightarrow} 0,
\]
uniformly on $k$ in probability, and there exists constant $C$ such that $|X^{(k)}_N|<C$ for all $k, N.$
Then 
\[
\frac{X^{(1)}_N+X^{(2)}_N+\dots X^{(N)}_N}{N}\overset{\Prob}{\rightarrow} 0.
\]
\item
Suppose that the array of random variables $\left(X^{(k, i)}_N\right)$ is such that
\[
X^{(k, i)}_N\underset{N \rightarrow \infty}{\rightarrow} 0,
\]
uniformly on $(k, i)$ in probability, and there exists constant $C$ such that $|X^{(k, i)}_N|<C$ for all $k, N.$
Then 
\[
Y_N^{(k)}=\frac{X^{(k, 1)}_N+X^{(k, 2)}_N+\dots X^{(k, N)}_N}{N}\underset{N \rightarrow \infty}{\rightarrow} 0.
\]
in probability uniformly on $k.$
\item Suppose that the array of random variables $X^{(1)}_N, X^{(2)}_N, \dots X^{(N)}_N$ is such that
\[
\frac{\sum_{k=1}^N{X^{(k)}_N}}{N}\underset{N \rightarrow \infty}{\rightarrow} 0,
\]
uniformly in probability, and there exists constant $C$ such that $|X^{(k)}_N|<C$ for all $k, N.$
\end{enumerate}
\label{lemma:convergenceofaverages}
\end{lemma}
\begin{proof}
For any $\epsilon>0$ there exist $N_0$ such that for all $N>N_0$ $$\max_{1\leq k \leq N}\Prob \left\{ |X^{(k)}_N|\geq \epsilon \right\} <\epsilon,$$
which leads to $\Exp\left|X_N^{(k)}\right|\leq C\epsilon +\epsilon,$ thus
\[
\Exp \left[\left|\frac{X^{(1)}_N+X^{(2)}_N+\dots+X^{(N)}_N}{N}\right|\right]\leq \epsilon C +\epsilon.
\]

and 

\[
\Prob \left[\left|\frac{X^{(1)}_N+X^{(2)}_N+\dots+X^{(N)}_N}{N}\right|>\sqrt{\epsilon}\right]\leq (C + 1) \sqrt{\epsilon},
\]
which proves the first part of the Lemma. Also, the second part of the Lemma can be derived from the computations above. 
\end{proof}

\begin{lemma}[Cauchy inequality]
\label{lemma:cauchy_inequality}
Suppose that a function $f(z, w)$ is analytic both in $z, w$ for $z, w \in \C \backslash \R$
Then 
\begin{equation}
    \left|\frac{\partial}{\partial z}\frac{\partial}{\partial w} f(z, w)\right| \leq 4 \frac{1}{|\Im z||\Im w|} \max_{\substack{|\tilde{z}-z|=0.5|\Im z|\\ |\tilde{w}-w|=0.5|\Im w|}} \left|f\left(\tilde{z}, \tilde{w}\right)\right|.
\end{equation}
\end{lemma}
\begin{proof}
    By Cauchy integral theorem 
    $$
        \frac{\partial}{\partial z}\frac{\partial}{\partial w} f(z, w)=\frac{1}{-4 \pi^2} \iint_{\substack{|\tilde{z}-z|=0.5|\Im z|\\ |\tilde{w}-w|=0.5|\Im w|}} \frac{f(\tilde{z}, \tilde{w})}{(\tilde{z}-z)^2(\tilde{w}-w)^2} d \tilde{z} \tilde{w}.
    $$
    Thus, 
    \begin{multline*}
        \left|\frac{\partial}{\partial z}\frac{\partial}{\partial w} f(z, w)\right|\leq \\ \frac{1}{4 \pi^2} \times \pi\left|\Im z\right|\times \pi\left|\Im z\right|\times\frac{\max_{\substack{|\tilde{z}-z|=0.5|\Im z|\\ |\tilde{w}-w|=0.5|\Im w|}} \left|f\left(\tilde{z}, \tilde{w}\right)\right|}{0.25 |\Im z|^2 \times 0.25 |\Im w|^2} = \\ 
        4 \frac{1}{|\Im z||\Im w|} \max_{\substack{|\tilde{z}-z|=0.5|\Im z|\\ |\tilde{w}-w|=0.5|\Im w|}} \left|f\left(\tilde{z}, \tilde{w}\right)\right|
        .
    \end{multline*}
\end{proof}

\begin{lemma}
    \label{lemma:analytic_convergence}
    Suppose that $f^{(k)}_N(z, w)$ the sequence of analytic on $z, w \in \C \backslash \R$ random functions such that for all fixed $z, w \in \C \backslash \R$
    $$
        f^{(k)}_N(z, w)\underset{N \rightarrow \infty}{\rightarrow}0
    $$
    uniformly on $k$ and for all $k, N$ with probability $1$ for all $z, w \in \C \backslash \R$
    $$
        \left|f^{(k)}_N(z, w)\right|\leq D(z, w)
    $$
    where $D(z, w)$ is a continuous function on $z, w \in \C \backslash \R.$ Then for all $z, w \in \C \backslash \R$
    $$
        \frac{\partial}{\partial z}\frac{\partial}{\partial w}f^{(k)}_N(z, w)\underset{N \rightarrow \infty}{\rightarrow}0
    $$
    uniformly on k in probability. 
    
\end{lemma}
\begin{proof}
    There exists a constant $M$ such that $|f_{N, k}(\tilde{z}, \tilde{w})|\leq M$ for $\tilde{z}: |\tilde{z}-z|\leq 0.8|\Im z|$ and  $\tilde{w}: |\tilde{w}-w|\leq 0.8|\Im w|.$ 
    Thus, there exists $C$ such that for all $z_1, z_2: |z_i-z|\leq 0.75|\Im z|$ and  $w_1, w_2: |w_i-w|\leq 0.75|\Im w|$ 
    \begin{equation}
        \left|f(z_1, w_1)-f(z_2, w_1)\right|\leq C\left|z_1-z_2\right|
    \end{equation}
    and 
     \begin{equation}
        \left|f(z_1, w_1)-f(z_1, w_2)\right|\leq C\left|w_1-w_2\right|.
    \end{equation}
    For every $\epsilon>0$ choose a finite collection of $\tilde{z}_1, \tilde{z}_2 \dots \tilde{z}_p$  and $\tilde{w}_1, \tilde{w}_2 \dots \tilde{w}_q$, such that 
   \begin{multline}
       |\tilde{z}-z|=0.5 |\Im z|\textit{ and } |\tilde{w}-w|=0.5 |\Im w| \Rightarrow\\ \exists i<p, j<q : |\tilde{z}-\tilde{z}_i|<\epsilon \textit{ and } |\tilde{w}-\tilde{w}_j|<\epsilon
   \end{multline}    
 If $|\tilde{z}-\tilde{z}_i|<\epsilon \textit{ and } |\tilde{w}-\tilde{w}_j|<\epsilon$ thus, using the equation  we can get, that
   \begin{multline}
   |f^{(k)}_{N}(\tilde{z}, \tilde{w})-f^{(k)}_{N}(\tilde{z}_i, \tilde{w}_j)|\leq\\|f^{(k)}_{N}(\tilde{z}, \tilde{w})-f^{(k)}_{N}(\tilde{z}, w_j)|+|f^{(k)}_{N}(\tilde{z}, w_j)-f^{(k)}_{N}(\tilde{z}_i, \tilde{w}_j)|\leq 2 C \epsilon.
   \end{multline}
   There exists $N_0$ such that $\forall N > N_0$ $\max_{k}\Prob\left[\max_{i, j}|f_N^{(k)}(\tilde{z}_i, \tilde{w}_j)|>\epsilon \right]<\frac{\epsilon}{pq},$ thus for all $N>N_0,$ for all $k$
   \begin{equation}
       \max_k\Prob \left[\max_{\substack{|\tilde{z}-z|=0.5|\Im z| \\ |\tilde{w}-w|=0.5 |\Im w|}} \left|f^{(k)}_{N}(\tilde{z}, \tilde{w})\right|> 2C\epsilon +\varepsilon \right]<\epsilon.
   \end{equation}
   This way, using the Cauchy theorem, we get the statement of the Lemma. 
\end{proof}

\begin{lemma}[\cite{LM22}]
If $\Re(\sigma)<0,$
\begin{equation}
\int_{0}^{\infty} \frac{\exp (r \sigma)-r \sigma-1}{r^{\frac{\alpha}{2}+1}} \mathrm{~d} r = (-\sigma)^{\alpha/2} \times \Gamma(-\alpha/2)
\end{equation}\label{lemma:raising_into_power}
\end{lemma}

\begin{lemma}
If $\Re \sigma_1,\textit{ } \Re \sigma_1 > 0$
\begin{equation}
    \int_0^{\infty} \frac{e^{-\sigma_1 t} - e^{-\sigma_2 t}}{t} \diff t= -\log\sigma_1 + \log\sigma_2 = - \log\frac{\sigma_1}{\sigma_2}
\end{equation}
\label{lemma:subtract}
\end{lemma}

\begin{lemma}[\cite{LM22}]
 For $\sigma_1$ and $\sigma_2$ $ \in$ $\mathbb{C}\backslash \mathbb{R}$
 \begin{equation}
     \int_0^{\infty} \frac{r^{\frac{\alpha}{2} -1 }}{(r-\sigma_1)(r - \sigma_2 )} = \frac{\pi\left((-\sigma_1)^{\alpha/2} \sigma_2 - (-\sigma_2)^{\alpha/2}\sigma_1 \right)}{\sin\left(\frac{\pi\alpha}{2} \right)\sigma_2 \sigma_1 (\sigma_2 - \sigma_1)}.
 \end{equation} 
 \label{lemma:int_r}
 \end{lemma}

\section*{Acknowledgements}
I am very grateful to Asad Lodhia for his invaluable assistance with my paper. His expertise and guidance significantly improved the rigorousness of the proofs provided, and I cannot thank him enough for his kindness and support. Additionally, I would like to express my heartfelt gratitude to my supervisor, Anna Maltsev, for her continuous encouragement and feedback throughout the entire process. 

\bibliographystyle{alpha}
\bibliography{biblio}

\begin{thebibliography}{BGGM13}

\bibitem[AAP09]{auffinger_poisson_2009}
Antonio Auffinger, G{\'e}rard~Ben Arous, and Sandrine P{\'e}ch{\'e}.
\newblock {Poisson convergence for the largest eigenvalues of heavy tailed
  random matrices}.
\newblock {\em Annales de l'Institut Henri Poincar{\'e}, Probabilit{\'e}s et
  Statistiques}, 45(3):589 -- 610, 2009.

\bibitem[BGGM13]{benaych-georges_central_2014}
Florent Benaych-Georges, Alice Guionnet, and Camille Male.
\newblock Central limit theorems for linear statistics of heavy tailed random
  matrices.
\newblock {\em Communications in Mathematical Physics}, 329, 01 2013.

\bibitem[BGM16]{BGM16}
Florent Benaych-Georges and Anna Maltsev.
\newblock Fluctuations of linear statistics of half-heavy-tailed random
  matrices.
\newblock {\em Stochastic Process. Appl.}, 126(11):3331--3352, 2016.

\bibitem[BGT87]{bingham_regular_1989}
N.~H. Bingham, C.~M. Goldie, and J.~L. Teugels.
\newblock {\em Regular Variation}.
\newblock Encyclopedia of Mathematics and its Applications. Cambridge
  University Press, 1987.

\bibitem[BJYZ09]{bai_corrections_2009}
Zhidong Bai, Dandan Jiang, Jian-Feng Yao, and Shurong Zheng.
\newblock {Corrections to LRT on large-dimensional covariance matrix by RMT}.
\newblock {\em The Annals of Statistics}, 37(6B):3822 -- 3840, 2009.

\bibitem[Bor10]{borodin_clt_2014}
Alexei Borodin.
\newblock Clt for spectra of submatrices of wigner random matrices.
\newblock {\em Mosc. Math. J.}, 14, 10 2010.

\bibitem[BS98]{bai_no_1998}
Z.~D. Bai and Jack~W. Silverstein.
\newblock No eigenvalues outside the support of the limiting spectral
  distribution of large-dimensional sample covariance matrices.
\newblock {\em The Annals of Probability}, 26(1):316--345, 1998.

\bibitem[BS04]{bai_clt_2004}
Z.~D. Bai and Jack~W. Silverstein.
\newblock {CLT for linear spectral statistics of large-dimensional sample
  covariance matrices}.
\newblock {\em The Annals of Probability}, 32(1A):553 -- 605, 2004.

\bibitem[BS10]{BS10}
Zhidong Bai and Jack~W. Silverstein.
\newblock {\em Spectral analysis of large dimensional random matrices}.
\newblock Springer Series in Statistics. Springer, New York, second edition,
  2010.

\bibitem[BSY88]{bai_note_1988}
Z.D Bai, Jack~W Silverstein, and Y.Q Yin.
\newblock A note on the largest eigenvalue of a large dimensional sample
  covariance matrix.
\newblock {\em Journal of Multivariate Analysis}, 26(2):166--168, 1988.

\bibitem[BY88]{bai_convergence_1988}
Z.~D. Bai and Y.~Q. Yin.
\newblock {Convergence to the Semicircle Law}.
\newblock {\em The Annals of Probability}, 16(2):863 -- 875, 1988.

\bibitem[BY93]{bai_limit_1993}
Z.~D. Bai and Y.~Q. Yin.
\newblock {Limit of the Smallest Eigenvalue of a Large Dimensional Sample
  Covariance Matrix}.
\newblock {\em The Annals of Probability}, 21(3):1275 -- 1294, 1993.

\bibitem[DP18]{DP2018}
Ioana Dumitriu and Elliot Paquette.
\newblock Spectra of overlapping wishart matrices and the gaussian free field.
\newblock {\em Random Matrices: Theory and Applications}, 07(02):1850003, 2018.

\bibitem[EW16]{efthimiou_household_2016}
Costas~J. Efthimiou and Adam Wearne.
\newblock {Household income distribution in the USA}.
\newblock {\em The European Physical Journal B: Condensed Matter and Complex
  Systems}, 89(3):1--6, March 2016.

\bibitem[EY17]{erdos_dynamical_nodate}
László Erdös and Horng-Tzer Yau.
\newblock {\em A Dynamical Approach to Random Matrix Theory}.
\newblock American Mathematical Soc., 2017.

\bibitem[GGPS03]{gabaix_theory_2003}
Xavier Gabaix, Parameswaran Gopikrishnan, Vasiliki Plerou, and H.~Stanley.
\newblock A theory of power-law distributions in financial market fluctuations.
\newblock {\em Nature}, 01 2003.

\bibitem[LM22]{LM22}
Asad Lodhia and Anna Maltsev.
\newblock Covariance kernel of linear spectral statistics for half-heavy tailed
  wigner matrices.
\newblock {\em Random Matrices: Theory and Applications}, 0(0):2250054, 2022.

\bibitem[LP09]{lytova_central_2009}
A.~Lytova and L.~Pastur.
\newblock Central limit theorem for linear eigenvalue statistics of random
  matrices with independent entries.
\newblock {\em The Annals of Probability}, 37(5):1778--1840, 2009.

\bibitem[LRS20]{soshnikov_2020}
Lingyun Li, Matthew Reed, and Alexander Soshnikov.
\newblock Central limit theorem for linear eigenvalue statistics for
  submatrices of wigner random matrices.
\newblock {\em Frontiers in Applied Mathematics and Statistics}, 6, 2020.

\bibitem[LW02]{ledoit_hypothesis_2002}
Olivier Ledoit and Michael Wolf.
\newblock {Some hypothesis tests for the covariance matrix when the dimension
  is large compared to the sample size}.
\newblock {\em The Annals of Statistics}, 30(4):1081 -- 1102, 2002.

\bibitem[MM22]{MM22}
Anna~V. Maltsev and Svetlana Malysheva.
\newblock Corrigendum to ``{F}luctuations of linear statistics of
  half-heavy-tailed random matrices'' [{S}toch. {P}rocess. {A}ppl. 126 (2016)
  3331-3352].
\newblock {\em Stochastic Process. Appl.}, 146:414--415, 2022.

\bibitem[Shc11]{shcherbina_central_2011}
M.~Shcherbina.
\newblock Central limit theorem for linear eigenvalue statistics of the wigner
  and sample covariance random matrices.
\newblock {\em Zh. Mat. Fiz. Anal. Geom.}, 2011.

\bibitem[Sri05]{srivastava_tests_2005}
M.~Srivastava.
\newblock Some tests concerning the covariance matrix in high dimensional data.
\newblock {\em JOURNAL OF THE JAPAN STATISTICAL SOCIETY}, 35, 01 2005.

\bibitem[YBK88]{yin_limit_1988}
YQ~Yin, Z.~Bai, and P.~Krishnaiah.
\newblock On the limit of the largest eigenvalue of the large dimensional
  sample covariance matrix.
\newblock {\em Probability Theory and Related Fields}, 78:509--521, 08 1988.

\bibitem[YZB15]{yao_large_2015}
Jianfeng Yao, Shurong Zheng, and Zhidong Bai.
\newblock {\em Large Sample Covariance Matrices and High-Dimensional Data
  Analysis}.
\newblock Cambridge Series in Statistical and Probabilistic Mathematics.
  Cambridge University Press, 2015.

\end{thebibliography}

\end{document}